\documentclass[oneside,a4paper,reqno]{amsart}

\usepackage{enumerate}  

\usepackage{mathrsfs}   

\usepackage{multirow}

\usepackage{amsmath, amsthm, amscd, amssymb, latexsym, eucal, mathtools}
\usepackage{tensor}

\usepackage[all]{xy}

\def\serieslogo@{} \def\@setcopyright{} \makeatother

\usepackage{boondox-calo}

\usepackage{upgreek}   

\usepackage{graphicx}

\usepackage{array}

\usepackage{multienum}

\usepackage[colorinlistoftodos]{todonotes}

\usepackage{hyperref}
\usepackage{color}
\usepackage{cite}
\usepackage[capitalise, noabbrev, nameinlink]{cleveref}
\crefname{subsection}{Subsection}{Subsections}

\usepackage{tikz-cd}
\usepackage{quiver}

\usepackage{slashbox}

\usepackage{makecell}

\makeatletter
\renewcommand*\env@matrix[1][c]{\hskip -\arraycolsep
	\let\@ifnextchar\new@ifnextchar
	\array{*\c@MaxMatrixCols #1}}
\makeatother

\usepackage{color}

\usepackage{adjustbox}

\usepackage{hhline}

\usepackage{float}

\usepackage{tikz}
\usetikzlibrary{calc}

\makeatletter
\def\@endtheorem{\endtrivlist}
\makeatother

\newcommand{\C}{\mathscr{C}}
\newcommand{\D}{\mathscr{D}}

\newcommand{\La}{\Lambda}

\newcommand{\Ll}{\mathcal{l \! l}}
\newcommand{\lal}{\mathcal{l }}

\newcommand{\BQ}{ \mathcal{B}_{ Q} }
\newcommand{\BQA}{ \mathcal{B}_{\! Q}^{ A} }
\newcommand{\BQnotA}{ \mathcal{B}_Q^{\, \nott \!\! A} }

\newcommand{\epic}{\twoheadrightarrow}

\newcommand{\monicc}{\hookrightarrow}

\newcommand{\la}{\langle}
\newcommand{\ra}{\rangle}

\newcolumntype{C}[1]{>{\centering\arraybackslash}m{#1}}

\newcommand\TTT{\rule{0pt}{2.6ex}}       
\newcommand\BBB{\rule[-1.2ex]{0pt}{0pt}} 

\DeclareMathOperator{\pd}{\mathsf{pd}}
\DeclareMathOperator{\fd}{\mathsf{fd}}
\DeclareMathOperator{\idim}{\mathsf{id}}

\DeclareMathOperator{\id}{\mathrm{id}}

\DeclareMathOperator{\Hom}{\mathrm{Hom}}

\DeclareMathOperator{\Ker}{\mathrm{Ker} \!}
\DeclareMathOperator{\Image}{\mathrm{Im} \!}

\DeclareMathOperator{\rmod}{\! - \mathrm{mod}}

\DeclareMathOperator{\rMod}{\! - \mathrm{Mod}}

\DeclareMathOperator{\nott}{\mathrm{not}}

\DeclareMathOperator{\fpd}{\mathsf{fin.dim}}
\DeclareMathOperator{\Fpd}{\mathsf{Fin.dim}}

\DeclareMathOperator{\End}{\mathsf{End}}

\DeclareMathOperator{\rad}{\mathrm{rad}}
\DeclareMathOperator{\topp}{\mathsf{top}}

\DeclareMathOperator{\gd}{\mathsf{gl.dim}}

\DeclareMathOperator{\wgd}{\mathsf{w.gl.dim}}

\DeclareMathOperator{\Ext}{\mathrm{Ext}}
\DeclareMathOperator{\Tor}{\mathrm{Tor}}

\newtheorem{thm}{Theorem}[section]
\newtheorem{cor}[thm]{Corollary}
\newtheorem{lem}[thm]{Lemma}
\newtheorem{prop}[thm]{Proposition}

\newtheorem*{thmA}{Theorem A}
\newtheorem*{thmB}{Theorem B}

\theoremstyle{definition}

\newtheorem{defn}[thm]{Definition}

\newtheorem{exam}[thm]{Example}

\newtheorem{con}[thm]{Construction}

\newtheorem*{defnn}{Definition}

\theoremstyle{remark}

\newtheorem{rem}[thm]{Remark}

\def\a{\alpha}
\def\b{\beta}
\def\g{\gamma}
\def\d{\delta}

\def\i{\iota}

\hypersetup{hidelinks}

\begin{document}

	\title{Radical preservation and the finitistic dimension}

	\author[Giatagantzidis]{Odysseas Giatagantzidis}
	\address{O.~Giatagantzidis\\
		Dept of Mathematics\\
		AUTh\\
		54124 Thessaloniki\\
		Greece}
	\email{odysgiat@math.auth.gr}
	
	\date{\today}

	

	\begin{abstract}
		We introduce the notion of radical preservation and prove that a radical-preserving homomorphism of left artinian rings of finite projective dimension with superfluous kernel reflects the finiteness of the little finitistic, big finitistic and global dimension. As an application, we prove that every bound quiver algebra with quasi-uniform Loewy length, a class of algebras introduced in this paper, has finite (big) finitistic dimension. The same result holds more ge\-ne\-rally in the context of semiprimary rings. Moreover, we construct an explicit family of such finite dimensional algebras where the finiteness of their big finitistic dimension does not follow from existing results in the literature.
	\end{abstract}

	\maketitle

	\vspace*{-0.6cm}

	\setcounter{tocdepth}{1} \tableofcontents{}

	\phantom{asdf}
	
	{}\vspace*{-0.8cm}

	\section{Introduction and main results}

	\smallskip
	
	One of the most important homological dimensions of a ring $R$ is the \emph{little finitistic dimension} introduced by Auslander and Buchsbaum \cite{findimdef}{}, denoted by $\fpd R$, which is defined to be the supremum of the projective dimensions of finitely generated \emph{left} modules with finite projective dimension. 
	Its usefulness lies on the fact that it provides a more accurate measure for the homological complexity of the category of finitely generated modules compared to the global dimension, denoted by $\gd R$, when the latter is infinite.

	A few years later, Bass \cite{Bass}{} publicized the question of Rosenberg and Zelinsky  whether the little finitistic dimension is finite for every ring. We know now that there are rings with infinite little finitistic dimension, with one of the first examples given by Nagata in the early 1960's in the context of commutative noetherian rings, see for example \cite{Krause}{}. Another example was given by Kirkman and Kuzmanovich \cite{KirkKuzm}{} in the context of semiprimary non-noetherian rings. 
	However, the question remains open for Artin algebras and it has been promoted to a conjecture, namely the \emph{Finitistic Dimension Conjecture $\mathsf{(FDC)}$}, often considered in the context of finite dimensional algebras. We refer to the survey \cite{Huisgentale}{} for an overview of classes of Artin algebras that were known to satisfy the $\mathsf{(FDC)}$ up to 1995.

	Another relevant dimension which has received its fair share of attention is the big finitistic dimension due to Kaplansky (see introduction of \cite{Bass}{}), denoted by $\Fpd R$, which resembles the little one with the difference that the supremum is taken over all left modules (not necessarily finitely generated) with finite projective dimension. Huisgen \cite{Huisgen, Huisgen2}{} proved that this dimension can be strictly larger than its little counterpart by considering appropriate monomial bound quiver algebras.

	In this paper, we introduce the notion of radical preservation which turns out to be a key notion for the establishment of new reduction results for the finiteness of the finitistic dimensions of Artin algebras, such as \hyperlink{thmA}{Theorem A} of this paper and the results of \cite{Giata2}{}.
	Moreover, we introduce the class of bound quiver algebras with quasi-uniform Loewy length, see definition below, and prove that every such algebra satisfies the $\mathsf{(FDC)}$.

	A ring homomorphism $\phi \colon A \to B $ is called \emph{radical-preserving} if the image of the Jacobson radical of $A$ under $\phi $ is contained in the Jacobson radical of $B$.
	Furthermore, the kernel of $\phi $ is called \emph{superfluous} if it is contained in the Jacobson radical of $ A $, and $\phi $ is \emph{of finite projective dimension} if the projective dimension of $B$ viewed as a \emph{right} $ A $-module via restriction of scalars along $ \phi $ is finite.

	Our first main result restricted to the class of Artin algebras is the following.
	
	\begin{thmA}[\cref{cor:1}{}]	\phantomsection	\hypertarget{thmA}
		If $\phi \colon A \to B$ is a radical-preserving homomorphism of Artin algebras with superfluous kernel, then it holds that
			\begin{equation*}
				\fpd A \leq \fpd B + \pd B_A .
			\end{equation*}
		Moreover, the analogous inequalities hold for the big finitistic and global dimensions of the algebras.
	\end{thmA}

	We remark that special instances of radical-preserving monomorphisms, like radical embeddings \cite{EHIS}{} and radical-full monomorphisms \cite{FDCandfingldim}{}, have been studied before.

	It is also worth mentioning that a similar result to the above exists for surjective ring homomorphisms under fewer assumptions on the rings, see \cite[Theorem~1]{Small}{} and \cite[Theorem~1.8]{KKS}{}. 
	However, the technique of \cite{Small}{} cannot be employed if one drops the surjectivity assumption; see end of \cref{sec:rad.pres.homs}{} for details.
	In contrast to that technique, we compute module projective dimensions via minimal projective resolutions instead of arbitrary projective resolutions of minimal length. For this reason, \hyperlink{thmA}{Theorem A} is actually proven in the contexts of left noetherian semiperfect rings (for $\fpd $ and $\gd $) and of left perfect rings (for $\Fpd $ and $\gd $).

	An important additional feature of radical-preserving homomorphisms is that they are abundant and extend the class of surjective ring homomorphisms significantly. In particular, we show that every homomorphism $\phi \colon A \to B$ of Artin algebras is radical-preserving whenever $B$ is basic, see \cref{propbasic}{}.
	As every Artin algebra is Morita equivalent to its basic version, we deduce that
	\emph{the $\mathsf{(FDC)}$ holds for Artin algebras if and only if for every basic Artin algebra $A$ there exists a homomorphism $\phi \colon A \to B$ with superfluous kernel of finite projective dimension, where $B$ is a basic Artin algebra that satisfies the $\mathsf{(FDC)}$}.
	
	Clearly, it is not an easy task to establish such a homomorphism in general.	
	Nonetheless, for any bound quiver algebra $\La $, we construct an algebra $\La\!^* $ with $ \Fpd \La^* = 0 $, related to $\La $ via a radical-preserving monomorphism $ \i \colon  \La \monicc \La\!^*$; see \cref{con:q.unif.Ll.1}{} and \hyperlink{lem:q.unif.Ll.2}{\cref{lem:q.unif.Ll.2}{}}{}.
	Furthermore, the monomorphism $ \i $ is of finite projective dimension precisely when $\La $ has \emph{quasi-uniform Loewy length} (see \cref{rem:1}{}), a property that we introduce as follows.

	\begin{defnn}[\cref{defn:q.unif.Ll.1}{}]
			A semiprimary ring $ \La $ has quasi-uniform Loewy length if the Loewy length of every indecomposable projective module whose top has infinite injective dimension is maximal, i.e.\ equal to the Loewy length of $ \La $.
	\end{defnn}

	Our second main result, proven for bound quiver algebras alternatively by applying \hyperlink{thmA}{Theorem A} to the monomorphism $ \i \colon \La \monicc \La\!^*$, is as follows.

	\begin{thmB}[\cref{thm:B}{}, \cref{thm:q.unif.Ll}{}, \cref{cor:2}{}]			\phantomsection	\hypertarget{thmB}
		For every semiprimary ring $ \La $ with quasi-uniform Loewy length, it holds that $ \Fpd \La $ is finite. If $ \La $ is in particular an Artin algebra, then $ \Fpd \La $ is bounded above by $ \fpd \La^{ \! \mathrm{op} } $.
	\end{thmB}

	Furthermore, if the global dimension of a semiprimary ring is finite, then it follows from \cref{thm:B}{} that it is equal to the injective dimension of a simple module whose projective cover has non-maximal Loewy length; see \cref{cor:q.unif.Ll}{}.

	Last, we construct an infinite two-parameter family of bound quiver algebras with quasi-uniform Loewy length, see \cref{exam:q.unif.Ll}{}, where the finiteness of their big finitistic dimension does not follow from existing results in the literature. Although the finiteness of their little finitistic dimension can be deduced from the vertex removal operation \cite{arrowrem1}{}, we show that our bound is arbitrarily smaller for the family. To the best of our knowledge, our example is the first one showing that the difference $ \fpd \La ^{ \! \mathrm{op}} - \Fpd \La $ can be arbitrarily big for non-monomial bound quiver algebras $\La $ with non-zero finitistic dimensions, see \cref{cor:q.unif.Ll.1}{}.

	We close this introduction by outlining the contents of the paper. In \cref{sec:preliminaries}{}, we collect standard facts about semiperfect rings and useful implications thereof.
	In \cref{sec:rad.pres.homs}{}, we characterize radical preservation in homological terms (see \cref{prop:reverse}{}) and prove \hyperlink{thmA}{Theorem A}, applied subsequently to extend several existing results.
	\cref{sec:algebras.of.q.unif.Ll}{} is devoted to the proof of \hyperlink{thmB}{Theorem~B} in its full generality, and in the specific context of bound quiver algebras through the radical-preserving monomorphism $ \i \colon  \La \monicc \La\!^*$.
	We conclude the paper with \cref{exam:q.unif.Ll}{} and \cref{cor:q.unif.Ll.1}{} mentioned above.

	\subsection*{Notation}

	We denote by $J( R)$ the Jacobson radical of an associative ring $R$ with unit. A module over a ring will be a \emph{left} module unless stated otherwise.
	For a module ${}_R M$, we denote by $\rad _R M$ its radical and by $\topp_R M$ the induced quotient $M / \rad _R M $. By $\pd {}_R M$, $\fd {}_R M$ and $\idim {}_R M$ we denote the projective, flat and injective dimension of $ M $, respectively. Similarly, we write $N_R$ to denote a right $R$-module, its radical is denoted by $\rad N_R $ and so on.
	We denote by $R \rMod$ (resp.\ $R \rmod$) the category of left (finitely generated) $R$-modules. The respective right module categories are denoted by $ R^{\mathrm{op}} \rMod $ and $ R^{\mathrm{op}} \rmod $. The little finitistic, big finitistic and global dimension of $R$ are denoted by $\fpd R$, $\Fpd R$ and $\gd R$, respectively.
	If $R$ is a semiprimary ring, then $\Ll (R )$ denotes its Loewy length. Similarly, we denote by $\Ll({}_R M)$ the Loewy length of a module ${}_R M$.

	\subsection*{Acknowledgements}
	
	I would like to thank my Ph.D.\ supervisor, Chrysostomos Psa\-rou\-da\-kis, for several useful discussions regarding this paper. I would also like to express my gratitude to Steffen Koenig for giving me the opportunity to present the first part of the paper in the Algebra Seminar of the University of Stuttgart
	in the winter semester of 2023. Finally, my appreciation goes to the anonymous referee, who pointed out the broader applicability of \cref{thm:q.unif.Ll} to the context of Artin algebras and offered a helpful sketch of proof.

	The present research project was supported by the Hellenic Foundation for
	Research and Innovation (H.F.R.I.) under the ``3rd Call for H.F.R.I.\ Ph.D.\
	Fellowships'' (F.N.: 47510/03.04.2022).

	\medskip
	
	\section{Preliminaries on semiperfect rings}	\label{sec:preliminaries}

	\smallskip

	In this section, we collect well-known facts about semiperfect rings and derive useful implications thereof. We refer the reader to \cite{AndersonFuller}{} for more details.
	
	A ring $R$ is called \emph{semiperfect} if any of the following equivalent conditions holds.
	\begin{enumerate}[\rm(i)]
		\item $R$ is semilocal and idempotents lift modulo $J( R)$.
		\item Every finitely generated $R$-module possesses a projective cover.
		\item Every simple $R$-module possesses a projective cover.
		\item There is a decomposition of $R$ into a direct sum of local $R$-modules.
	\end{enumerate}
	Recall that $R$ is called \emph{semilocal} if it is semisimple modulo its Jacobson radical.
	
	We remark that semiperfectness is a left-right symmetric notion, as is evident from condition (i), and condition (iv) readily implies that every local ring is semiperfect being a local module over itself.
	
	A finite subset $\{ e_i \}_{i \in I} $ of a semiperfect ring $R$ is called a \emph{complete set of primitive orthogonal idempotents} if the following hold. (i) Every $e_i$ is a primitive idempotent, that is $e_i^2 =e_i$ and the left ideal $Re_i$ is indecomposable; (ii) it holds that $e_{i_1} e_{i_2} = 0 $ for every pair of distinct indices $i_1$ and $i_2 $; (iii) $  \sum_i e_i =  1_R $. Moreover, the set $\{ e_i \}_{i \in I'} $ for a subset $I' \subseteq I $ is called a \emph{basic set of primitive orthogonal idempotents} if for every index $i \in I$ there is a unique index $ i ' \in I ' $ such that $ { {}_R Re_i } \simeq { {}_R R e_{i'} } $.
	
	A semiperfect ring $R$ is called \emph{basic} if a complete set of primitive orthogonal idempotents for $R$ is also basic.
	In the following lemma, we relate that property with a connection between the Jacobson radical of $R$ and the set $N(R)$ of nilpotent elements in $R$. Recall that a subset of $R$ is called \emph{nil} if it is contained in $N(R)$.

	\begin{lem}
		\label{lem:0.basic.semiperfect.rings}
		Let $R$ be a semiperfect ring. If $R$ is basic, then $ N(R) $ is contained in $ J( R) $. Conversely, ring $R$ is basic if the two sets are equal. Therefore, if $J(R)$ is nil, ring $R$ is basic if and only if $ J( R) = N(R) $.
	\end{lem}
	
	\begin{proof}
		If $ \{ e_i \}_{ i \in I }$ is a complete set of primitive orthogonal idempotents for $R$, then $ { R / J } = \oplus_{i \in I} ( { R / J } ) ( { e_i + I } )$ is a decomposition of $R / J$ into a direct sum of simple {$R/J$}-modules as idempotents lift modulo $J = J(R)$. Assume that $I' $ is a subset of $I$ such that $\{ { e_i + J } \}_{ i \in I'}$ is a basic set of primitive orthogonal idempotents for $R / J$ and set $m = | I' |$. According to the Wedderburn-Artin structure theorem for semisimple rings, there are division rings $D_1, \ldots, D_m$ such that $R /J$ is the direct product
		\[
		R / J \simeq  \times_{j=1}^m \mathbb{M}_{\kappa_j} (D_j) 
		\]
		where $\kappa_j$ is the number of indices $i \in I$ such that $ { ( { R/J } ) ( { e_i + J } ) } \simeq { ( { R/J } ) ( { e_j + J } ) } $ for every $j \in I'$, and $\mathbb{M}_{\kappa_j} (D_j)$ is the full ring of $\kappa_j \times \kappa_j$ matrices over $D_j = \End_R ( { R e_j / J e_j } )^{\mathrm{op}} $. 
		
		It holds that $R$ is basic if and only if $R/J$ is basic, see \cite[Proposition~17.18]{AndersonFuller}{}. Therefore, if $R$ is basic then $ I' = I $ and $\kappa_j = 1$ for every $j$. In particular, the ring $R / J $ cannot contain any non-zero nilpotent elements as it is isomorphic to the direct product of $|I|$ division rings. As a consequence, if $x \in N(R)$ then $x + J$ is nilpotent in $R / J$ implying that $x \in J$. Conversely, assume that $J = N(R)$ and $R$ is not basic. Then $\kappa_j > 1$ for some  $ j \in I' $ and, therefore, the ring $\mathbb{M}_{\kappa_j} (D_j)$ contains non-zero nilpotent elements, for example the matrix with first row $(0 \,\, 1 \,\, 0 \,\, \ldots \,\, 0)$ and zero in every other position. We deduce that the ring $R / J$ contains non-zero nilpotent elements. If $x + J$ is such an element, there is an integer $\nu >1 $ such that $x^\nu \in J$ and, since $J$ is nil, it holds that $x \in N(R) = J$, a contradiction.
	\end{proof}
	
	For the sake of completeness, we prove the following fact which is mentioned in \cite[p.~303]{AndersonFuller}{} and will be needed in the next section.
	
	\begin{lem}		\label{lem:0.idempotents}		\hypertarget{lem:0.idempotents}
		For a semiperfect ring $R$, the Jacobson radical is the unique largest ideal of $R$ that contains no non-zero idempotents.
	\end{lem}
	
	\begin{proof}
		For any ring $R$ it holds that its Jacobson radical $J = J(R)$ does not contain any non-zero idempotents, see \cite[Corollary~15.11]{AndersonFuller}{}. Let now $I$ be an ideal that does not contain any non-zero idempotents and assume that $I $ is not contained in $ J$. Then $(I + J) / J$ is a non-zero submodule of ${}_R R / J$. Since ${}_R R /J$ is semisimple and idempotents lift modulo $J$, there is a non-zero idempotent $e \in I + J$ which can be assumed to be primitive by considering a simple direct summand of ${}_R (I + J) / J$ for example. Let $e = x + y$ for $x \in I$ and $y \in J$. Then we may assume without loss of generality that $x = e x e $ and $y = e y e $, and it follows that $e - e y e \in I$ and $e y e \in e J e = J(e R e)$. Since $e R e$ is a local ring, we deduce that $e x e = e - e y e$ is invertible in $e R e$ implying that $e \in I$, a contradiction.
	\end{proof}

	An important subclass of semiperfect rings are left perfect rings. A ring $R$ is called \emph{left perfect} if it satisfies any of the following equivalent conditions. 
	\begin{enumerate}[\rm(i)]
		\item Every left $R$-module possesses a projective cover.
		\item $R$ is semilocal and $J( R)$ is left T-nilpotent.
		\item Every flat left $R$-module is projective.
	\end{enumerate}
	Recall that a subset $ T \subseteq R$ is called \emph{left T-nilpotent} if for every sequence $ ( a_i )_{i \in \mathbb{N}} $ in $ T $ there is a positive integer $n$ such that $a_1 a_2 \ldots a_n = 0$.
	
	
	The notion of T-nilpotence is not left-right symmetric and one may define analogously right perfect rings.
	A ring that is both left and right perfect is called \emph{perfect}. 
	Moreover, every left or right perfect ring is semiperfect with nil Jacobson radical.
	
	The following lemma condenses well-known results in a key form, vital for the proof of \hyperlink{thmA}{Theorem~A}.
	
	\begin{lem}
			\label{lem:2}
		Let $ R $ be a ring and $ {}_R P $ a projective module. If $ P $ is finitely generated, then a submodule of $ P $ is superfluous if and only if it is contained in $ \rad_R P = J( R ) P $. Moreover, the following are equivalent:
		\begin{enumerate}[\rm(i)]
			\item A submodule of $ P $ is superfluous if and only if it is contained in $ \rad_R P = J(R) P $ for any (not necessarily finitely generated) $ P $.
			\item The Jacobson radical of $ R $ is left T-nilpotent.
		\end{enumerate}
	\end{lem}

	\begin{proof}
		For any module ${}_R M $ and a superfluous submodule $ N $, it is a straightforward standard fact that $ N \subseteq \rad_R M $.
		As $ \rad _R P = J( R ) P $ for any projective module (see for instance \cite[Proposition~17.10]{AndersonFuller}{}), it follows that a superfluous submodule of $ P $ is contained in $ J ( R ) P $ in any case.
		
		For the other direction of the first claim, recall that $J( R ) M $ is a superfluous submodule of $ M $ for every finitely generated module $ {}_R M $ according to Nakayama's Lemma (see \cite[Corollary~15.13]{AndersonFuller}{}). The claim follows now by taking $ M = P $ to be a finitely generated projective module, since a submodule of a superfluous submodule is itself superfluous in the initial module.
		
		According to \cite[Lemma~28.3]{AndersonFuller}{}, the Jacobson radical of $ R $ is left T-nilpotent if and only if $ J( R ) M $ is a superfluous submodule of $ M $ for any module $ {}_R M $. Therefore, one implication of our second claim is straightforward.
		
		For the other implication, assume that for any projective module $ P $ (not necessarily finitely generated), every submodule contained in $ J( R ) P $ is superfluous. Equivalently, we assume that $ J( R ) P $ is a superfluous submodule of $ P $ for any $ P $, and it suffices to show that $ J( R ) M $ is a superfluous submodule of $ M $ for every module ${}_R M $. Since there is an epimorphism $ P \epic M $ of $ R $-modules for some projective $ {}_R P $, we may identify $ M $ with $ P / K $ where $ K $ is the kernel of the epimorphism. Note that $ J( R ) M $ corresponds to $ ( J( R ) P + K ) / K $ under the implied identification. Therefore, it suffices to show that $ ( J( R ) P + K ) / K $ is a superfluous submodule of $ P / K $ for any submodule $ K $ of a projective module $ P $. Let $ { ( J( R ) P + K ) / K } + { L / K } = P / K $ for a submodule $ L / K $ of $ P / K $. Equivalently, we have that $ J ( R ) P + L = P $ as $ K \subseteq L $, implying that $ L = P $ as $ J ( R ) P $ is a superfluous submodule of $ P $ by assumption. This completes the proof.
	\end{proof}

	We close this section with an observation that rests on well-known results. Its proof is included for the convenience of the reader, as the lemma will be employed in the next section.
	
	\begin{lem}
			\label{lem:1}
		A ring $ R $ is left noetherian left perfect if and only if it is left artinian.
	\end{lem}

	\begin{proof}
		Recall that a ring is left artinian if and only if it is left noetherian, semilocal and its Jacobson radical is a nilpotent ideal, a result due to Hopkins (see for instance \cite[Theorem~15.20]{AndersonFuller}{}). This already implies one direction of the statement, as a semilocal ring with nilpotent Jacobson radical is a special case of a (left) perfect ring. For the other direction, recall that a nil (one-sided) ideal of a left noetherian ring is nilpotent according to a result of Levitzki (see \cite[Theorem~15.22]{AndersonFuller}{}). Therefore, the Jacobson radical of a left noetherian left perfect ring is nilpotent as it is nil, and Hopkins' Theorem applies.
	\end{proof}

	\smallskip

	\section{Radical-preserving homomorphisms} \label{sec:rad.pres.homs}

	\smallskip

	The main objective of this section is to prove \hyperlink{thmA}{Theorem~A} from the introduction in its most general form. Throughout, we fix a ring homomorphism $ \phi \colon A \to B$. We start by introducing radical preservation, the central notion of the paper.

	\begin{defn}
		A ring homomorphism $\phi \colon A \to B $ is called \emph{radical-preserving} if $\phi( J(A) ) \subseteq J(B)$.
	\end{defn}
	
	For the sake of succinctness, we say that the kernel of $\phi$ is \emph{superfluous} if it is contained in $ J( A )$, and $\phi$ will be called \emph{of finite (right) flat dimension} if the flat dimension of $B$ as a right $A$-module via restriction of scalars along $\phi $ is finite. Similarly, we say that $ \phi $ is \emph{of finite (right) projective dimension} if $ \pd B_A < \infty $.

	\begin{rem}
		The term ``superfluous'' comes from the fact that a left (or right) ideal of $A$ is a superfluous submodule of $A$ if and only if it is contained in $ J( A )$, see \cite[Theorem~15.3]{AndersonFuller}{}. 
	\end{rem}
	
	We introduce yet another term that will simplify the exposition.

	\begin{defn}
		A non-zero module ${}_A M$ is called \emph{$\phi$-vanishing} if $B \otimes_A M = 0$ or, else, it is called \emph{$\phi$-nonvanishing}. Similarly, a non-zero homomorphism $f \in A \rMod$ is called \emph{$\phi$-vanishing} if $ B \otimes_A f = 0$ or, else, it is called \emph{$\phi$-nonvanishing}.
	\end{defn}

	In our first lemma, we show that the induction functor $ B \otimes_A -$ along $\phi $ preserves projective covers in a sense that is made precise in \cref{defn:rad.pres}{}, especially when $J(B)$ is a left T-nilpotent ideal of $B$. We refer the reader to \cref{lem:2}{}, as it is essential for the proof of the following lemma.

	\begin{lem}		\label{lem:1.partial.preservation.of.proj.covs}	\phantomsection		\hypertarget{lem:1.partial.preservation.of.proj.covs}{}
		Let $ \phi \colon A \to B $ be a radical-preserving ring homomorphism. Let ${}_A M$ be a non-zero module with projective cover $f \colon P \epic M$. Assume furthermore that
		\begin{enumerate}[\rm(i)]
			\item ${}_A M$ is finitely generated, or
			\item $J(B)$ is a left T-nilpotent ideal of $B$.
		\end{enumerate}	
		Then $M$ is $\phi$-vanishing if and only if $ P $ is $\phi$-vanishing, if and only if $ f $ is $\phi$-vanishing. Moreover, the morphism $ B \otimes_A f$ is a projective cover of $ B \otimes_A M$ when the latter is non-zero.
	\end{lem}
	
	\begin{proof}
		Let $ \i \colon \Ker f \monicc P$ denote the inclusion homomorphism. It follows from the right exactness of the functor $B \otimes_A -$ that the sequence
		\[
		B \otimes_A \Ker f \xrightarrow{ B \otimes_A \i } B \otimes_A P \xrightarrow{ B \otimes_A f} B \otimes_A M \to 0
		\]
		is exact. In particular, it holds that $B \otimes_A P = 0$ implies $B \otimes_A M = 0$. We have that $\Ker f \subseteq \rad_A P = J(A) P$, since it is a superfluous submodule of $P$, and  
		\[
			{ \Ker \, ( { B \otimes_A f } ) } = { \Image \, ( { B \otimes_A \i } ) } \leq { B \otimes { J(A) P } }  \leq { J(B) \otimes P } = { J(B) ( { B \otimes_A P } ) }
		\]
		where the last module is equal to the radical of $ B \otimes_A P$. It follows that $ \Ker \, ( B \otimes_A f) $ is a superfluous submodule of $ B \otimes_A P$ if either $ M$ is finitely generated (in which case $ P$ and $ B \otimes_A P$ are also finitely generated) or $J(B)$ is left T-nilpotent, see \cref{lem:2}{}. If $B \otimes_A M = 0$ then $\Ker \, ( B \otimes_A f) = B \otimes_A P$, implying that $B \otimes_A P  $ is a superfluous submodule of itself, which may happen only if $B \otimes_A P = 0$.
	\end{proof}

	\begin{rem}
		An analog of \hyperlink{lem:1.partial.preservation.of.proj.covs}{\cref{lem:1.partial.preservation.of.proj.covs}{}} can be proven in the same way for arbitrary superfluous epimorphisms, that is epimorphisms whose kernel is a superfluous submodule; see also \cite[Corollary~15.13 (Nakayama's Lemma), Lemma~28.3]{AndersonFuller}{}. The only necessary additional assumption is that $A$ is semilocal so that $\rad_A M = J(A) M$ for every module over $A$, see \cite[Corollary~15.18]{AndersonFuller}{}.
	\end{rem}

	The following definition will be employed for the characterization of radical preservation from a homological point of view. We also use the analogous terms for injective envelopes in \cref{cor:rad.pres}{}.

	\begin{defn}
		\label{defn:rad.pres}
		Let $F \colon \C \to \D $ be a covariant functor of abelian categories and let $X$ be an object in $\C $. We say that \emph{$F$ preserves the projective cover of $X$} if there is a projective cover $f \colon P \to X $ in $\C$ and either (i) morphism $F(f) $ is a projective cover of the non-zero object $ F (X) $ in $ \D $, or (ii) both $F(X)$ and $F(P)$ are zero. We say that $F$ preserves the projective cover of $X$ \emph{non-trivially} if (i) holds.
	\end{defn}

	With the above terminology in mind, the first part of \hyperlink{lem:1.partial.preservation.of.proj.covs}{\cref{lem:1.partial.preservation.of.proj.covs}{}} says that the induction functor along a radical-preserving ring homomorphism preserves the projective covers of finitely generated modules. We show next that the notion of radical preservation is in fact minimal with respect to that property.

	\begin{prop}
		\label{prop:reverse}
		Let $\phi \colon A \to B$ be a ring homomorphism. The induction functor along $ \phi $ preserves the projective covers of finitely generated $A$-modules if and only if $\phi$ is radical-preserving.
	\end{prop}
	
	\begin{proof}
		One implication is \hyperlink{lem:1.partial.preservation.of.proj.covs}{\cref{lem:1.partial.preservation.of.proj.covs}{}}.
		For the converse implication, note that the natural epimorphism $f \colon A \epic A / J( A )$ is a projective cover of $ A / J( A )$ viewed as an $ A $-module. Furthermore, the following diagram is clearly commutative
		\[\begin{tikzcd}
			{ B \otimes_A A} & { B \otimes_A  { A / J(A) } } \\
			{ B} & { B / { B \phi ( J(A) ) } }
			\arrow["{ \! \! \! B \otimes_A f}", two heads, from=1-1, to=1-2]
			\arrow[shorten <=5pt, two heads, from=2-1, to=2-2]
			\arrow["g"', from=1-1, to=2-1]
			\arrow["h", from=1-2, to=2-2]
			\arrow["\simeq", from=1-1, to=2-1]
			\arrow["\simeq"', from=1-2, to=2-2]
		\end{tikzcd}\]
		where $g( { b \otimes a } ) = { b \phi(a) } $ and $h( { b \otimes ( { a + J( A ) } ) } ) = { { b \phi(a) } + { B J( A ) } }$ for every $a \in A$ and $b \in B$, and the lower horizontal map is the natural epimorphism. Since the module $ B \otimes_A A \simeq B $ is non-zero, the homomorphism $ B \otimes_A f$ is a projective cover by assumption. In particular, it holds that $ { B \phi( J( A ) ) } $ is a superfluous submodule of ${}_B B$, implying that  $ { \phi(J( A )) } \subseteq { B \phi ( J( A ) ) } \subseteq { J( B ) } $.
	\end{proof}

	Next, we characterize the property ``the kernel of $\phi$ is superfluous'' in terms of the nonvanishing of finitely generated projective $A$-modules under the induction functor when $A$ is semiperfect.
	
	\begin{lem}
		\label{lem:2.kernel.in.radical}
		Let $\phi \colon A \to B$ be a ring homomorphism where $A$ is semiperfect. Then the kernel of $\phi$ is superfluous if and only if every finitely generated projective $A$-module is $\phi$-nonvanishing.
		If this is the case and $A$ is additionally left perfect, then all projective $A$-modules are $\phi$-nonvanishing.
	\end{lem}
	
	\begin{proof}
		Assume that $A$ is semiperfect. Let $\{ e_i \}_i $ be a basic set of primitive orthogonal idempotents for $ A $, and recall that the set $\{ e_i \}_i $ is finite by definition. For every finitely generated projective module ${}_A P$ we have
		\phantomsection		\hypertarget{eqq}{}
		\begin{equation}	
			\label{eq:1.projective.modules}
			P \simeq \bigoplus_{i}  A e_i ^{(T_i)}
		\end{equation}
		where $T_i$ is a finite set and $ A e_i ^{(T_i)}$ is the direct sum of $|T_i|$ copies of $ A e_i$ for every $i$; see \cite[Characterization~27.13]{AndersonFuller}{}. It is easily verified that $ B \otimes_A A e_i \simeq  B \phi(e_i)$ through the $B$-isomorphism given by $b \otimes a e_i \mapsto b \phi (a e_i)$ with inverse $ { b \phi(e_i) } \mapsto { { b \phi (e_i) } \otimes e_i } $ for all $a \in A$ and $b \in B$. Therefore, we have
		\[
		B \otimes_A P \simeq \bigoplus_{i} ( B \otimes_A A e_i) ^{(T_i)} \simeq \bigoplus_{i}  B \phi( e_i) ^{(T_i)}
		\]
		and $P$ is $\phi$-vanishing if and only if $ e_i  \in \Ker \phi$ for all $i$ such that $T_i$ is non-empty.
		
		For any ring $A$, the radical $J(A)$ does not contain any non-zero idempotents. Therefore, the inclusion $\Ker \phi \subseteq J(A)$ implies that every finitely generated projective $A$-module is $\phi$-nonvanishing.
		On the other hand, if $\Ker \phi \nsubseteq J(A)$ then there is a non-zero idempotent $e \in \Ker \phi$ according to \hyperlink{lem:0.idempotents}{\cref{lem:0.idempotents}{}}, implying that the finitely generated projective module $ A e$ is $\phi$-vanishing.
		For the last part of the lemma, it suffices to recall that all projective $ A $-modules are direct sums as in \hyperlink{eqq}{(\ref{eq:1.projective.modules}{})} by allowing the sets $T_i$ to be infinite if $A$ is left perfect; see \cite[Proposition~28.13]{AndersonFuller}{}.
	\end{proof}

	\begin{cor}
		\label{cor:rad.pres}
		Let $\phi \colon A \to B$ be a homomorphism of Artin algebras. Then the following are equivalent.
		\begin{enumerate}[\rm(i)]
			\item The homomorphism $\phi$ is radical-preserving (with superfluous kernel).
			\item The induction functor $B \otimes_A - $ preserves projective covers (non-trivially).
			\item The induction functor $B \otimes_A - $ preserves projective covers of finitely generated modules (non-trivially).
			\item The coinduction functor $ \Hom _A ( B , - ) $ preserves injective envelopes of fi\-ni\-tely generated modules (non-trivially).
		\end{enumerate}
	\end{cor}
	
	\begin{proof}
		As Artin algebras are perfect rings, the equivalence between (i) and (ii) (or (iii)) is a direct consequence of \hyperlink{lem:1.partial.preservation.of.proj.covs}{\cref{lem:1.partial.preservation.of.proj.covs}{}}{} for the preservation part and \cref{lem:2.kernel.in.radical}{} for the non-triviality part.
		The equivalence between (iii) and (iv) is due to the fact that the restriction of the coinduction functor $ \Hom _A ( B , - ) $ on finitely generated modules is naturally isomorphic to the dual induction functor $ { - \otimes_A B } \colon A^{\mathrm{op}} \rmod  \to B^{\mathrm{op}} \rmod $ up to the standard dualities between left and right finitely generated modules over $A$ and $B$; see for instance \cite[Theorem~3.1.(b), Theorem~3.3]{repth}{}.
	\end{proof}

	Next, we establish the key property of radical-preserving homomorphisms with superfluous kernel.
	A module $ {}_A M$ is called \emph{$\phi$-flat} if $\Tor^A_i(B , M) = 0$ for all $i \geq 1$.

	\begin{prop}[cf.\!\mbox{\cite[Lemmas~1~and~1']{Small}{}}]		\phantomsection		\hypertarget{prop:radical-preserving.implies.preservation.of.proj.covs}{}
		\label{prop:radical-preserving.implies.preservation.of.proj.covs}
		Let $\phi \colon A \to B $ be a radical-pre\-ser\-ving ring homomorphism with superfluous kernel where $A$ is semiperfect.
		\begin{enumerate}[\rm(i)]
			\item The induction functor along $ \phi $ preserves non-trivially the projective covers of finitely generated $A$-modules. Furthermore, we have $\pd {}_A M = \pd {}_B B \otimes_A M$ for every finitely generated $\phi$-flat module $ M $ if $A$ is in addition left noetherian.
			\item If $A$ is left perfect and $J(B)$ is a left T-nilpotent ideal of $B$, then the induction functor along $ \phi $ preserves non-trivially all projective covers. Moreover, it holds that $\pd {}_A M = \pd {}_B B \otimes_A M$ for every $\phi$-flat module $ M $.
		\end{enumerate}
	\end{prop}

	\begin{proof}
		The first parts of (i) and (ii) follow directly from \hyperlink{lem:1.partial.preservation.of.proj.covs}{\cref{lem:1.partial.preservation.of.proj.covs}{}} and \cref{lem:2.kernel.in.radical}{}. If $A$ is semiperfect and left noetherian, then every finitely generated module ${}_A M$ has a minimal projective resolution, say $\mathbb{P}$, where all projective modules are finitely generated. If ${}_A M$ is also $\phi$-flat, then the complex $ B \otimes_A \mathbb{P}$ is exact and, thus, a minimal projective resolution of $ B \otimes_A M$ according to the first part of (i). Similarly, if $A$ is left perfect and $J(B)$ is left T-nilpotent, then every module ${}_A M$ has a minimal projective resolution $\mathbb{P}$, and the complex $ B \otimes_A \mathbb{P}$ is a minimal projective resolution of $ B \otimes_A M$ if ${}_A M$ is $\phi$-flat. In both cases, the desired equality holds as the complexes $ B \otimes_A \mathbb{P} $ and $ \mathbb{P} $ have the same length.
	\end{proof}

	We are now ready to prove the main theorem of this section.

	\begin{thm}		\hypertarget{thm:main}{}
		\label{thm:main}
		Let $\phi \colon A \to B $ be a radical-preserving ring homomorphism with superfluous kernel.
		\begin{enumerate}[\rm(i)]
			\item If $A$ is left noetherian semiperfect, then
			\[
			\fpd A \leq \fpd B + \fd B_A.
			\]
			\item If $A$ is left perfect and $J(B)$ is a left T-nilpotent ideal of $B$, then
			\[
			\Fpd A \leq \Fpd B + \fd B_A.
			\]
		\end{enumerate}
		In both cases, it holds that $\gd A \leq \gd B + \fd B_A$.
	\end{thm}

	\begin{proof}
		We assume that $\fd B_A = d $ is finite, as there is nothing to prove otherwise. Let $M$ be an $A$-module with minimal projective resolution
		\[
			\cdots \to P_n \xrightarrow{f_n} \cdots \to P_1 \xrightarrow{f_1} P_0 \xrightarrow{f_0} M \to 0
		\]
		and let $\Omega^i_A(M)$ denote the kernel of $f_{i-1}$ for all $i \geq 0$, where $f_{-1} : M \to 0 $ is the trivial homomorphism. We have $\Tor^A_i (B, \Omega_A^d (M)) \simeq \Tor^A_{d+i} (B, M) \simeq 0$ for any $i \geq 1$, where the first isomorphism is commonly called `dimension shift', and the second Tor-group is trivial as it can be calculated from a flat resolution of $B_A$. In other words, the $ d $-th syzygy of any $ A $-module is $ \phi $-flat.

		We assume now that $A$ is left noetherian semiperfect and take a finitely generated module ${}_A M$. All modules occurring in the minimal projective resolution of $ M$, including the kernels, are finitely generated. In particular, the module $ \Omega_A^d (M)$ is finitely generated and $\phi$-flat, implying that
		\[
		\pd _A M = d + \pd _A \Omega_A^d(M) = d + \pd _B B \otimes_A \Omega_A^d(M) \leq d + \gd B
		\]
		if $\pd _A M > d$. Indeed, tensoring with $ B_A $ preserves the projective dimension of finitely generated $ \phi $-flat modules according to \hyperlink{prop:radical-preserving.implies.preservation.of.proj.covs}{\cref{prop:radical-preserving.implies.preservation.of.proj.covs}{}.(i)}{}. Since $ M $ was an arbitrary finitely generated $ A $-module and the global dimension of a ring may be computed on finitely generated modules (see \cite[Part~III, Section~14]{Kaplansky}{}), taking the supremum yields $ \gd A \leq \gd B + d $.
		
		Restricting our attention to modules $ {}_A M $ as above with $\pd _A M < \infty $, we get
		\[
		\pd _A M \leq d + \fpd B
		\]
		as $B \otimes_A \Omega_A^d(M)$ is a finitely generated $B$-module of finite projective dimension (equal to $\pd _A M - d$) in this case. Indeed, it holds that $ \pd _B ( B \otimes_A \Omega_A^d(M) ) \leq \fpd B $ in this case by the definition of the little finitistic dimension, and taking the supremum yields $ \fpd A \leq \fpd B + d $.
		
		The proof of part (ii) is analogous, as all $ A $-modules have minimal projective resolutions when $ A $ is left perfect and in view of \hyperlink{prop:radical-preserving.implies.preservation.of.proj.covs}{\cref{prop:radical-preserving.implies.preservation.of.proj.covs}{}.(ii)}{}.
	\end{proof}

	The following corollary is a direct consequence of \hyperlink{thm:main}{\cref{thm:main}{}}{}, as a ring is left noetherian left perfect if and only if it is left artinian; see \cref{lem:1}{}.

	\begin{cor}
			\label{cor:1}
		Let $\phi \colon A \to B $ be a radical-preserving homomorphism of left artinian rings with superfluous kernel. Then
			\[
			\fpd A \leq \fpd B + \pd B_A
			\]
		and the inequality remains valid if we replace $ \fpd $ with $ \Fpd $ or $ \gd $.
	\end{cor}

	As an Artin algebra is always left artinian as a ring (see for instance \cite[Section~II.1]{repth}{}), we obtain the following corollary.

	\begin{cor}
		\label{corartinalgs}
		A radical-preserving homomorphism of Artin algebras of finite projective dimension with superfluous kernel reflects the finiteness of the little finitistic, big finitistic and global dimension.
	\end{cor}
	
	The condition of radical preservation is automatically satisfied in many cases, as shown in the next proposition. In particular, every Artin algebra homomorphism is radical-preserving if the target algebra is basic.

	\begin{prop}
		\label{propbasic}
		A ring homomorphism $\phi \colon A \to B $ is radical-preserving in the following cases:
		\begin{enumerate}[\rm(i)]
			\item The ring homomorphism $ \phi $ is surjective.
			\item The Jacobson radical of $A$ is nil and $B$ is basic semiperfect.
		\end{enumerate}
	\end{prop}
	
	\begin{proof}
		For (i), see \cite[Corollary~15.8]{AndersonFuller}{}. The second part follows from \cref{lem:0.basic.semiperfect.rings}{}, which implies that nilpotent elements of $ B $ are contained in its Jacobson radical, and the fact that a ring homomorphism preserves nilpotency of elements.
	\end{proof}
	
	As a consequence of \hyperlink{thm:main}{\cref{thm:main}{}}{}, \cref{propbasic}{} and the fact that every Artin algebra is Morita equivalent to a basic Artin algebra, we obtain the following equivalent reformulation of the Finitistic Dimension Conjecture $\mathsf{(FDC)}$ for Artin algebras. Note that the reformulation remains valid if Artin algebras are replaced by finite dimensional algebras over a field.
	
	\begin{cor}
		\label{restateFDC}
		The $\mathsf{(FDC)}$ holds for Artin algebras if and only if for every basic Artin algebra $ A $ there exist
		\begin{enumerate}[\rm(i)]
			\item a basic Artin algebra $ B $ with $  \fpd B  < \infty$, and
			\item a homomorphism $\phi \colon A \to B $ with superfluous kernel of finite projective dimension.
		\end{enumerate}
	\end{cor}

	We close this section by applying \hyperlink{thm:main}{\cref{thm:main}{}}{} in order to (partly) recover and extend results from the literature.
	We begin with a classic result due to Small and Kirkman-Kuzmanovich-Small, restricted to the class of left artinian rings.
	
	\begin{cor}[cf.\!\mbox{\cite[Theorem~1]{Small}{}},\!\mbox{\cite[Theorem~1.8]{KKS}{}}]
		For a left artinian ring $ A $ and a superfluous ideal $ K $, it holds that
		\[
		\fpd A \, \leq \, \fpd { A / K } + \pd {A / K}_A .
		\]
		Moreover, the analogous inequalities hold for the big finitistic and global dimensions of the rings.
	\end{cor}

	\begin{proof}
		Direct application of \hyperlink{thm:main}{\cref{thm:main}{}}{}.
	\end{proof}
	
	The crucial result proven in \cite{Small}{} is the equality of projective dimensions $\pd _A M = \pd _B B \otimes_A M$ for every $\phi$-flat module ${}_A M$, where $\phi \colon A \epic B$ is a surjective ring homomorphism such that its kernel is a nilpotent ideal of $A$. The crux of the proof is to show that $M$ is $A$-free if the module $B \otimes_A M$ is $B$-free. More specifically, it was shown that if a subset $ T $ of the module $  B \otimes_A M $ is a free $B$-basis, then any preimage of $ T $ under the natural $A$-epimorphism $ \eta_M \colon M \epic {}_A B \otimes_A M $, sending every $m \in M $ to $ 1_B \otimes m $, is a free $A$-basis of $M$.
	Similarly, the same equality of projective dimensions was proven for finitely generated $\phi $-flat modules under the extra assumption that $A$ is left noetherian, whereas the kernel of $\phi $ is allowed to be just superfluous in that case.

	When $\phi$ is no longer surjective, the natural map $M \to {{}_A B \otimes_A M } $ fails to be surjective in general and, therefore, the method of \cite{Small}{} cannot be employed in order to prove the above equality of projective dimensions for $\phi$-flat modules. \cref{prop:radical-preserving.implies.preservation.of.proj.covs}{} shows that this obstacle can be overcome when $\phi$ is radical-preserving under natural extra assumptions on the rings.

	The next result concerns a class of radical-preserving monomorphisms of Artin algebras studied in \cite{FDCandrelprojmods}{}.

	\begin{cor}[\mbox{\!\!\cite[Corollary~1.4]{FDCandrelprojmods}{}}]		\label{cor:rad.pres.1}		\hypertarget{cor:rad.pres.1}
		Let $ \phi \colon A \monicc B$ be a monomorphism of Artin algebras such that $ \phi ( J(A) )$ is a left ideal of $B$ and the projective dimension of $B$ as a right $A$-module is finite. Then $\fpd A \leq \fpd B + \pd B_A + 2$.
	\end{cor}
	
	\begin{proof}
		It holds that $ \phi $ is radical-preserving as $ \phi ( J(A) ) $ is a nilpotent left ideal of $B$, see  \cite[Corollary~15.10]{AndersonFuller}{}. The result follows now from \hyperlink{thm:main}{\cref{thm:main}{}}{} and the fact that $B$ is a perfect ring.
	\end{proof}

	It should be noted that \hyperlink{thm:main}{\cref{thm:main}{}}{} yields in fact the smaller upper bound $ \fpd B + \pd B_A $ for $ \fpd A$ in the setup of \hyperlink{cor:rad.pres.1}{\cref{cor:rad.pres.1}{}}, as well as the analogous upper bounds for $ \Fpd A $ and $ \gd A $.

	More recently, it was proven in \cite{bifiniteexts}{} that if $ \phi \colon A \monicc B $ is a monomorphism of finite dimensional algebras such that the quotient $B / A$ has finite projective dimension as an $A$-bimodule, that is as a module over the enveloping algebra of $A$, then the finiteness of $\fpd B $ implies the finiteness of $\fpd A $. Under those conditions, the quotient $B / A$ has in particular finite projective dimension as a right $A$-module and the same holds for $B$ due to the short exact sequence $0 \to A \to B \to B/A \to 0$ of right $A$-modules. If in addition $ B $ is basic, then it follows from \cref{propbasic}{} that $ \phi $ is radical-preserving. In particular, we may recover the above implication through \hyperlink{thm:main}{\cref{thm:main}{}}{} and also extend it for the other two dimensions of $A$ and $B$.

	The last result concerns the (left) global dimension of left perfect and left noetherian semilocal rings.

	\begin{cor}[cf.\!\mbox{\cite[Corollary~11]{Ausgldim}{}}]
		\label{corgldim1}
		If $R$ is a left perfect or left noetherian semilocal ring, then its weak global dimension is attained among the flat dimensions of simple right modules. In particular, it holds that
		\[
		\gd R = \max \{ \fd S_R \ | \ S_R \textrm{ simple}\} .
		\]
	\end{cor}

	\begin{proof}
		In both cases $R$ is semilocal, implying that $\gd { R / J(R) } = 0$ and $ { R / J(R) } _R$ is a semisimple module containing every simple right $R$-module up to isomorphism as a direct summand. In particular, it holds that
		\[
		\fd \, {R / J(R)}_R = \max \{ \fd S_R \ | \ S_R \textrm{ simple} \}.
		\]
		Furthermore, it holds that $\gd R = \wgd R$, where $\wgd R$ denotes the weak global dimension of $R$. Indeed, for a left perfect (resp.\ left noetherian) ring, a (finitely generated) left module is projective if and only if it is flat, implying that the projective and flat dimension of a (finitely generated) left module coincide.
		
		Applying \hyperlink{thm:main}{\cref{thm:main}{}}{} for the radical-preserving (see \cref{propbasic}{}) natural epimorphism $R \epic R / J(R)$  if $R$ is left perfect, or \cite[Theorem~1]{Small}{} if $R$ is left noetherian semilocal, yields
		\[
		\gd R \leq \fd {R / J(R)}_R
		\]
		and the result follows from the fact that $ \fd { R / J(R) }_R \leq \wgd R$.
	\end{proof}

	\begin{rem}
		Note that \cref{corgldim1}{} recovers part of \cite[Corollary~2.7]{Mochi}{}.
		Moreover, if we assume $R$ to be left perfect, then its global dimension is attained among the injective dimensions of simple left modules (see \cite[Lemma~13]{Osofsky}{}), implying that $\max \{ \idim {}_R L \ | \ {}_R L \textrm{ simple}\} = \max \{ \fd S_R \ | \ S_R \textrm{ simple}\}$.
	\end{rem}

	\smallskip
	
	\section{Algebras of quasi-uniform Loewy length}		\label{sec:algebras.of.q.unif.Ll}
	
	\smallskip

	The main aim of this section is to prove \hyperlink{thmB}{Theorem B}{} from the introduction. We begin with a well-known fact about semiprimary rings $ \La $ such that every indecomposable direct summand of the regular module ${}_\La \La $ has Loewy length equal to the Loewy length of the ring. We say that such a ring has \emph{uniform Loewy length (from the left)}, and provide a generalization later on.
	
	\begin{prop}	\phantomsection		\hypertarget{prop:friendly}{}
		\label{prop:friendly}{}
		The big finitistic dimension of a semiprimary ring with uniform Loewy length is zero.
	\end{prop}

	\begin{proof}
		Let $ \La $ be a semiprimary ring with uniform Loewy length, and note that $ \La $ is perfect as it is semilocal and its Jacobson radical is nilpotent. In particular, all $ \La $-modules have a projective cover. According to \cref{lem:2}{}, this means that for every non-zero $ \La $-module $ M $ there is an epimorphism $ f \colon P \epic M $ where ${}_\La P $ is projective and the kernel $ K = \Ker f $ is contained in $ \rad {}_\La P = J( \La ) P $. It follows that the Loewy length of $ K $ is strictly smaller than the Loewy length of $P$, implying that $ K $ is not projective as all projective $ \La $-modules have the same Loewy length. Indeed, every projective module is isomorphic to a direct sum of direct summands of the regular module; see \cite[Proposition~28.13]{AndersonFuller}{}. Therefore, if $ M $ is not projective, then $ K $ is a non-zero non-projective $ \La $-module. We deduce that $ \pd _\La M = \infty $ by iterating this process. All in all, we have shown that a $ \La $-module is either projective or has infinite projective dimension, and the proof is complete.
	\end{proof}

	We go on to fix some notation before introducing the key concept of this section. For any semiprimary ring $ \La $, we let $ \{ e_i \}_i  $ be a fixed complete set of primitive orthogonal idempotents. Furthermore, we denote by $e_\infty $ the sum of all idempotents $e_i$ such that the injective dimension of $ \topp_\La \La e_i $ is infinite. Last, we say that the Loewy length of a $ \La $-module $ M $ is \emph{maximal} if it is equal to the Loewy length of $ \La $, since $ \Ll (  M ) \leq \Ll ( \La ) $ in general.

	\begin{defn}
		\label{defn:q.unif.Ll.1}
		A semiprimary ring $ \La $ has \emph{quasi-uniform Loewy length (from the left)} if the Loewy length of every indecomposable projective module occurring as a direct summand of $ {}_\La \La e_\infty $ is maximal.
	\end{defn}

	We remark that among semiprimary rings uniform Loewy length and finite global dimension are the two extremes for quasi-uniform Loewy length. Indeed, a semiprimary ring $\La$ has quasi-uniform Loewy length if and only if for every idempotent in $\{  e_i \}_i $ it holds that (i) the Loewy length of $\La e_i$ is maximal or (ii) the injective dimension of $ \topp_\La \La e_i $ is finite. Therefore, the Loewy length of $\La$ is uniform exactly when condition (i) holds for every idempotent in $\{  e_i \}_i $, while the global dimension of $\La$ is finite exactly when condition (ii) holds for every idempotent (see \cite[Corollary~11]{Ausgldim}{}).

	The following lemma is a short detour from the main goal of this section.
	
	\begin{lem}
		Let $ \La $ be an Artin algebra of uniform Loewy length from the left. Then the Loewy length of $ \La $ is quasi-uniform from the right if and only if it is uniform from the right.
	\end{lem}

	\begin{proof}
		Let us assume that the Loewy length of $\La $ is uniform from the left, and quasi-uniform but not uniform from the right. Then there is a primitive idempotent $e_i$ such that the Loewy lenght of $ e_i \La _\La $ is not maximal but the injective dimension of $ \topp e_i \La_\La $ is finite. It follows that the projective dimension of $ \topp_\La \La e_i $ is finite, since $D ( \topp e_i \La_\La ) \simeq \topp_\La \La e_i $ for the standard duality $D$ between finitely generated left and right $\La$-modules. Consequently, the uniformity of the Loewy length of $\La $ from the left implies that $ \topp_\La \La e_i $ is projective (see \hyperlink{prop:friendly}{\cref{prop:friendly}{}}{}), which is possible only if $ J( \La ) \La e_i = \rad_\La \La e_i = 0 $ since the natural epimorphism $ \La e_i \epic \topp_\La \La e_i $ is a projective cover. We deduce that $\La $ is semisimple as $\Ll (\La ) = \Ll ( \La e_i ) = 1 $, implying that its Loewy length is uniform from both sides, a contradiction. The other direction follows by definition, completing the proof.
	\end{proof}

	We are now ready to prove the main result of the section.

	\begin{thm}
			\label{thm:B}
		Let $ \La $ be a semiprimary ring. Then $ \Fpd \La $ is at most the supremum of the injective dimensions of simple $ \La $-modules $ S $ such that the projective cover of $ S $ has non-maximal Loewy length. In particular, if $ \La $ has quasi-uniform Loewy length, then $ \Fpd \La $ is finite.
	\end{thm}

	\begin{proof}
		Let $ {}_\La M $ be a module of finite projective dimension $ \pd_\La M = n > 0 $. We begin by showing that $ \Ext_\La^n ( M , \topp_\La P_n ) \neq 0 $, where
			\[
				0 \to P_n \xrightarrow{f_n} P_{n-1} \to \cdots \to P_1 \xrightarrow{f_1} P_0 \xrightarrow{f_0} M \to 0
			\]
		is a minimal projective resolution of $ M $. The availability of such a resolution is due to $ \La $ being a perfect ring (with nilpotent Jacobson radical).
		
		By definition, we have
			\[
				\Ext_\La^n ( M , \topp_\La P_n ) \simeq \Hom_\La ( P_n , \topp_\La P_n ) / \Image f_n^*
			\]
		where $ f_n^* = \Hom_\La ( f_n , \topp_\La P_n ) $ sends a $ \La $-homomorphism $ g \colon P_{n-1} \to \topp _\La P_n $ to $ f_n^* ( g ) = g f_n $. But $ \Image f_n $ is equal to the kernel of $ f_{n-1} $, which is a projective cover, implying that it is a superfluous submodule of $ P_{n-1} $. Therefore $ \Image f_n $ is contained in $ J( \La ) P_{n-1} $ according to \cref{lem:2}{}. Furthermore, we have $ g( J( \La ) P_{n-1} ) = 0 $ since $ J( \La ) \topp_\La P_n = J( \La ) ( P_n / J( \La ) P_n ) = 0 $. We deduce that $ f_n^*(g) = 0 $ for every $ g $, that is $ \Image f_n^* = 0 $. It remains to observe that $ \Hom_\La ( P_n , \topp_\La P_n ) $ is non-trivial, since it contains the natural epimorphism for instance.
		
		The fact that $ \Ext_\La^n ( M , \topp_\La P_n ) $ is non-trivial implies that $ \idim \topp_\La P_n \geq n $. But $ \topp_\La P_n $ is a direct sum of simple modules since $ \La $ is semilocal, and the projective covers of these simple modules have non-maximal Loewy length as they are direct summands of $ P_n $. Indeed, the Loewy length of $ P_n $ is bounded above by the Loewy length of $ J( \La ) P_{ n-1 } $ due to the embedding $ f_n $, and $ \Ll (  J( \La ) P_{ n-1 } ) = \Ll ( P_{ n-1 } ) -1 $. The first assertion is now evident.
		
		For the last assertion, observe that the Loewy length of $ \La $ is quasi-uniform if and only if the established upper bound for $ \Fpd \La $ is finite. Indeed, assume that the Loewy length of $ \La $ is quasi-uniform and $ S $ is a simple $ \La $-module with projective cover $ P $ whose Loewy length is not maximal. Let $ e_i $ be the primitive idempotent such that $ P \simeq \La e_i $, implying that $ S \simeq \topp_\La \La e_i $. Since the Loewy length of $ \La e_i $ is not maximal, it holds that $ \La e_i $ cannot occur as a direct summand of $ \La e_\infty $ and thus $ \idim_\La S < \infty $. The first implication follows now from the fact that there is a finite number of simple $ \La $-modules up to isomorphism. Conversely, assume that for every simple $ \La $-module $ S $ whose projective cover $ P $ has non-maximal Loewy length, it holds that $ \idim_\La S < \infty $. Now let $ P = \La e_i $ be a direct summand of $ \La e_\infty $ for some primitive idempotent $ e _i $, that is $ \idim _\La S = \infty $ for $ S = \topp _\La \La e_i $. Therefore, it holds that the Loewy length of $ \La e_i $ is maximal by assumption as the natural epimorphism $ \La e_i \epic \topp_\La \La e_i $ is a projective cover, which completes the proof.
	\end{proof}

	\begin{cor}
		\label{cor:q.unif.Ll}
		Let $ \La $ be a semiprimary ring with finite global dimension. Then its global dimension is attained among the injective dimensions of simple $ \La $-modules whose projective cover has non-maximal Loewy length.
	\end{cor}
	
	\begin{proof}
		Let $ S $ be a simple $ \La $-module whose projective cover has non-maximal Loewy length, chosen so that $ \idim _\La S $ is maximal. Then $ \gd \La = \Fpd \La \leq \idim _\La S \leq \gd \La $, where the first inequality follows from \cref{thm:B}{}.
	\end{proof}

	\begin{rem}
		It is evident from the proof of \cref{thm:B}{} that the established upper bound for $ \Fpd \La $ can be sharpened by considering only the simple $ \La $-modules $ S $ whose projective cover occurs as a direct summand of the last term in the minimal projective resolution of a $ \La $-module $ M $ with $ \pd _\La M < \infty $. This possibly smaller upper bound makes sense even when the Jacobson radical of the ring is not nilpotent, allowing thus for the following generalization.
	\end{rem}

	\begin{cor}
		Let $ \La $ be a left perfect ring. Then $ \Fpd \La $ is at most the supremum of the injective dimensions of simple $ \La $-modules $ S $ such that the projective cover of $ S $ occurs as a direct summand of the last term in the minimal projective resolution of a $ \La $-module $ M $ with $ \pd_\La M < \infty $. Moreover, an analogous upper bound holds for $ \fpd \La $ if $ \La $ is left noetherian semiperfect, by requiring additionally $ M $ to be finitely generated.
	\end{cor}

	\begin{proof}
		The proof for both cases is analogous to the proof of \cref{thm:B}{}. In the first case, left perfectness ensures all properties employed there. Specifically, every $ \La $-module $ M $ possesses a minimal projective resolution, the superfluous submodules of $ P_{ n-1 } $ are exactly the ones contained in $ J( \La ) P_{ n-1 } $, and $ P_n $ is the projective cover of its top which is a semisimple module. In the second case, we assume that $ M $ is finitely generated, and therefore it possesses a minimal projective resolution where all projective modules $ P_i $ are finitely generated, whence \cref{lem:2}{} ensures the desired characterization of the superfluous submodules of $ P_{ n - 1 } $.
	\end{proof}

	Restricting our attention to Artin algebras we obtain one more corollary.

	\begin{cor}
			\label{cor:2}
		For an Artin algebra $ \La $ with quasi-uniform Loewy length, it holds that $ \Fpd \La $ is finite and bounded above by $ \fpd \La^{ \! \mathrm{op} } $.
	\end{cor}

	\begin{proof}
		Let $ S $ be a simple $ \La $-module whose projective cover has non-maximal Loewy length, chosen so that $ \idim _\La S $ is maximal. Of course $ \idim _\La S < \infty $ as the Loewy length of $ \La $ is quasi-uniform, and $ \Fpd \La \leq \idim _\La S $ according to \cref{thm:B}{}. For the second assertion observe that $ \idim _\La S = \pd D ( {}_\La S ) _\La \leq \fpd \La ^{ \! \mathrm{op} } $, where $ D $ is the standard duality between finitely generated left and right $ \La $-modules.
	\end{proof}

	We proceed by providing an alternative proof for \cref{thm:B}{} with a combinatorial flavor, in the context of bound quiver algebras.

	For a finite quiver $Q$, we denote the sets of vertices and arrows of $Q$ by $Q_0$ and $Q_1$, respectively. We write $\BQ$ to denote the set of all paths in $Q$, including the trivial paths denoted by $e_i$ for every vertex $i \in Q_0 $.
	The source and target of a path $ p $ are denoted by $s(p)$ and $t(p)$, respectively, and we write $p q $ for two paths $p , q$ to denote the concatenation of $p$ followed by $q $. Furthermore, we write $\a \colon i \to j$ to denote an arrow $\a$ such that $s(\a) = i $ and $t(\a) = j$.

	\begin{con}[Uniformization]
		\label{con:q.unif.Ll.1}
		Let $ \La = k Q / I$ be a bound quiver algebra with Loewy length $ l = \Ll (\La )$ and vertices labeled $ 1, 2, \ldots, n $. We assume that $ \Ll( \La e_i) = l $ if and only if $ m < i \leq n$, where $ 0 \leq m < n $, by relabeling the vertices if necessary.
		Let $Q^*$ denote the quiver that results from $Q$ if we add one extra loop $\lal_i$ at each vertex $ i \leq m$, that is $Q^* = Q \dot \cup B$ where $ B = \{\lal_i \colon i \to i \ | \ i = 1, \ldots, m\} $.
		Furthermore, let $I^*$ be the ideal of $k Q^*$ generated by $ I $ and the set
		\begin{equation*}
			T_{\mathrm{u}} = \{ \lal_i^{l} , \, \gamma \lal_i, \, \lal_i \delta \ | \ i = 1, \ldots, m \}
		\end{equation*}
		where $\gamma$ and $\delta $ range over all arrows of $Q$ with target $i$ or source $i$, respectively. The \emph{uniformization} of $\La $ is defined to be the algebra $ \La\!^* = k Q^* / I^*$.
	\end{con}

	The natural inclusion $Q \monicc Q^*$ justifies one to view the elements of $k Q$ as elements of $k Q^*$ without risk of confusion. In particular, we use $e_i$ to denote both the trivial path of $k Q$ and $k Q^*$ that corresponds to any vertex $i \in Q_0 $.

	Recall that every element $ z $ of the path algebra $ k Q $ can be written uniquely as a sum $ z = \sum_p \mu_p \cdot p $
	where $p$ ranges over $\BQ $ and only finitely many coefficients $\mu_p \in k$ are non-zero.
	We say that a path $p $ \emph{occurs} in $z$ if $\mu_p $ is non-zero.
	Similarly, path $p$ \emph{occurs} in a subset $ T \subseteq k Q$ if it occurs in some element of $ T $.
	Furthermore, path $ p $ is \emph{divided by} a path $ q $ if $p = p_1 q p_2 $ for paths $p_1, p_2 \in \BQ $. We also say that $p$ \emph{avoids} $ q $ if $p$ is not divided by $ q $.

	\begin{lem}		\hypertarget{lem:q.unif.Ll.1}{}
		\label{lem:q.unif.Ll.1}
		The algebra $ \La \!^*$ is a bound quiver algebra. Furthermore, its Loewy length is bounded above by the Loewy length of $ \La $.
	\end{lem}
	
	\begin{proof}
		Every path occurring in $I^*$ has length at least two as the same holds for every path occurring in $I$ or $T_{\mathrm{u}}$. Now let $p$ be a path in $Q^*$ of length equal to $ l = \Ll ( \La ) $. If $p$ avoids all loops in $B$, then $p \in I$ implying that $p \in I^*$. If $p$ is divided by some loop $\lal_i \in B $, then either $p = \lal_i^{ l } $ or it contains a subpath of the form $\gamma \lal_i$ or $\lal_i \delta $ for appropriate arrows $\gamma , \delta \in Q$, implying that $ p $ is divided by a path in $ T_{\mathrm{u}} $. Hence $p \in I^*$ in every case, which completes the proof.
	\end{proof}
	
	In what follows, for a set of arrows $A \subseteq Q_1 $, we denote by $\BQA$ the set of paths in $Q$ divided by at least one arrow in $A$, and $\BQnotA$ denotes the set of paths avoiding all arrows in $ A $.
	Furthermore, for any element $z \in k Q$, we write $z_A$ and $z_{\nott \! A}$ to denote the unique elements of the subspaces $ {}_k \la \BQA \ra $ and $ {}_k \la \BQnotA \ra $ of $ k Q $, respectively, such that $z = z_A + z_{\nott \! A}$.

	\begin{lem}		\label{lem:q.unif.Ll.2}		\hypertarget{lem:q.unif.Ll.2}{}
		The following statements hold for a  bound quiver algebra $\La = k Q / I$ and its uniformization algebra $\La \! ^* = k Q^* / I^*$.
		\begin{enumerate}[\rm(i)]
			\item An element $w \in k Q^*$ is in $ I^* $ if and only if $w_{\nott \! B} \in I$ and every path occurring in $w_B$ is divided by a path from $ T_{\mathrm{u}} $. In particular, it holds that $ I^* \cap k Q = I $.
			\item The Loewy length of $\La \! ^*$ is uniform and equal to the Loewy length of $ \La $.
			\item The inclusion $Q \monicc Q^*$, sending each vertex and arrow of $Q$ to itself, induces an injective algebra homomorphism $ \i \colon \La \monicc \La \! ^* $. Similarly, the projection $Q^* \epic Q$, sending each vertex and arrow of $Q$ to itself and every loop in $B$ to zero, induces a surjective algebra homomorphism $\pi \colon \La \! ^* \epic \La $. Moreover, it holds that $\pi \i = \id_{\La} $.
		\end{enumerate}
	\end{lem}

	\begin{proof}
		(i) One implication follows immediately from the definition of $I^*$. For the converse implication, we characterize the elements of $ I^* $. Any element $w \in I^*$ is a $k$-linear combination of elements of the form $u t v $ where $ t \in I \cup T_{\mathrm{u}} $ and $u$, $ v$ are paths in $Q^*$. We split these terms into two cases. The first case is when $ t \in I $ and both paths $ u $ and $ v $ are in the original quiver $ Q $. In this case, we have $ utv \in I $ since $ I $ is an ideal of $ k Q $.
		
		In the second case, either $ t \in T_{\mathrm{u}} $ or at least one of $ u $ and $ v $ is divided by a loop $ \lal_i \in B $. In every such instance, all paths occurring in $ utv $ contain a subpath from $ T_{\mathrm{u}} $. For example, if $ t \in I $ and $ u $ is divided by some loop $ \lal_i \in B $, then every path occurring in $ u t v $ contains a subpath of the form $ \lal_i \d $ for arrows $ \d \in Q $ with source $ i $, since every path occurring in $ t $ is in $ Q $ and has length at least two.
		
		It follows that for any $ w \in I^* $, the element $w_{\nott \! B}$ is a $k$-linear combination of elements $ u t v $ of the first case and thus in $I$, and $w_{B}$ is a $k$-linear combination of elements $ u t v $ where every path contains a subpath from $ T_{ \mathrm{u} } $. The proof of equality $ I^* \cap k Q = I $ is now straightforward. The inclusion $ I \subseteq I^* \cap k Q $ is trivial. For the converse, if $ w \in I^* \cap k Q $ we have $ w = w_{\nott \! B} \in I $, which completes the proof.

		(ii) We have to show that there is a non-zero path $ p $ in $ \La \! ^*$ with length $ l - 1 $ and target $i$ for every vertex $ i $, as we already know that $ \Ll( { \La \! ^* e_i } ) \leq l = \Ll ( \La ) $ from \hyperlink{lem:q.unif.Ll.1}{\cref{lem:q.unif.Ll.1}{}}{}. For $i > m$, we let $p$ be a non-zero path of $ \La $ of length $ l - 1$ and target $i$, whose existence is guaranteed by the assumption $\Ll ( { \La e_i } ) = l $. For $i \leq m$, we set $ p = \lal_i^{ l - 1 } $.
		In both cases, it follows from (i) that $p $ is non-zero in $\La \! ^*$.

		(iii) It is immediate that the inclusion $Q \monicc Q^*$ induces an injective algebra homomorphism $k Q \monicc k Q^* $, see for instance \cite[Theorem~II.1.8]{ASS}{}. Furthermore, this induces an injective algebra homomorphism $ \La \monicc \La\!^* $ as $ I^* \cap k Q = I $ according to (i). Similarly, the projection $Q^* \epic Q$ induces the surjective algebra homomorphism $k Q^* \epic k Q $ sending $ w $ to $ w_{\nott \! B } $ for every $w \in k Q^*$, and this induces a surjective algebra homomorphism $ \La\!^* \epic \La $ as $w_{\nott \! B} \in I $ for any $w \in I^*$. The equality $\pi \i = \id_\La $ follows directly from the definitions, completing the proof of the lemma.
	\end{proof}

	We proceed to reprove \cref{thm:B}{} in the context of bound quiver algebras utilizing the results of \cref{sec:rad.pres.homs}{}.
	For a bound quiver algebra $\La = k Q / I $, we denote by $f_{\mathsf{n.m}}$ the sum of all trivial paths such that the corresponding indecomposable projective modules have non-maximal Loewy length. In other words, we define $ f_{\mathsf{n.m}} = \sum_{ 1 \leq i \leq m} e_i $ in the setup of \cref{con:q.unif.Ll.1}{}.
	Furthermore, we write $S_\La ( {i} )$ and $S_{ \La ^{ \! \mathrm{op}} } ({ i })$ to denote the simple left and right $\La$-module corresponding to every vertex $i \in Q_0$, respectively.

	\begin{thm}
		\label{thm:q.unif.Ll}
		For every bound quiver algebra $ \La = k Q / I $, it holds that
		\[
		\Fpd \La \leq \idim \topp_\La { \La f_{\mathsf{n.m}} } .
		\]
		In particular, if $\La $ has quasi-uniform Loewy length, then $\Fpd \La $ is finite and bounded above by $ \fpd \La ^{ \! \mathrm{op }} $.
	\end{thm}
	
	\begin{proof}
		The theorem follows from the application of \hyperlink{thm:main}{\cref{thm:main}{}}{} to the injective algebra homomorphism $ \i \colon \La \monicc \La \! ^*$ of \hyperlink{lem:q.unif.Ll.2}{\cref{lem:q.unif.Ll.2}{}.(iii)}, which is evidently radical-preserving; see also \cref{propbasic}{}. Furthermore, it holds that $\Fpd \La \! ^* = 0 $ according to \hyperlink{prop:friendly}{\cref{prop:friendly}{}}{}, as \hyperlink{lem:q.unif.Ll.2}{\cref{lem:q.unif.Ll.2}{}.(ii)} ensures that $ \La\!^* $ has uniform Loewy length. Therefore, it remains to show that $ \pd { \La  \! ^* } _\La = \idim \topp_\La { \La f_{\mathsf{n.m}} } $.
		
		Let $\La $ and $\La \! ^*$ be as in \cref{con:q.unif.Ll.1}{}, and let $\pi \colon \La \! ^*  \epic \La $ be the surjective algebra homomorphism of \hyperlink{lem:q.unif.Ll.2}{\cref{lem:q.unif.Ll.2}{}.(iii)}. The kernel of $ \pi $ is equal to the ideal of $ { \La \! ^* } $ generated by the loops in $B$, denoted by $\la B + I^* \ra $. Therefore, we have $ { \La \! ^* } _\La \simeq \La_\La \oplus { \la B + I^* \ra }_\La $ due to the equality $ \pi \i = \id_\La $, where $\la B + I^* \ra $ is viewed as a right $\La $-module via restriction of scalars along $\i$. Furthermore, the set
		\[
		\{ \lal_i^j + I^* \ | \ i = 1, 2, \ldots , m , \,  j  = 1 , 2, \ldots , l - 1  \}
		\]
		where $ l = \Ll ( \La ) $, is a $k$-basis of the ideal $ \la B + I^* \ra $, since the paths $ \lal_i^j $ as above are the only paths in $Q^*$ divided by some loop in $B$ while avoiding the paths in $T_{\mathrm{u}}$; see \hyperlink{lem:q.unif.Ll.2}{\cref{lem:q.unif.Ll.2}{}.(i)}.
		Moreover, it holds that
		\[
		\la B + I^* \ra _\La \simeq \bigoplus_{ 1 \leq i \leq m} ( { S_{\La^{ \! \mathrm{op}}} ( {i} ) } ) ^{\oplus ( l - 1 ) } 
		\]
			because $\la B + I^* \ra J (\La ) = 0 $ and $\lal_i e_i = \lal_i $ for every loop $ \lal_i \in B $. Consequently,
		\[
			\pd { \La  \! ^* } _\La = \pd \la B + I^* \ra _\La = \idim \topp_\La \La f_{ \mathsf{n.m} }
		\]
		as $D( S_{\La ^{ \! \mathrm{op}}}({i}) ) \simeq S_{\La}({i}) $ for every vertex $i \in Q_0 $, where $D $ is the standard duality between finitely generated left and right $\La $-modules. \hyperlink{thm:main}{\cref{thm:main}{}}{} implies now the desired upper bound for $ \Fpd \La $.
		
		If $\La $ has quasi-uniform Loewy length, then the injective dimension of the simple module $ S_\La ( i ) $ is finite for every vertex $i \leq m $ because $\La e_i $, which is the projective cover of $S_\La ( i ) $, has non-maximal Loewy length by assumption. Therefore, it holds that $ \Fpd \La \leq \idim \topp_\La \La f_{ \mathsf{n.m} } < \infty  $. Moreover, the established upper bound for $ \Fpd \La $ is equal to $ \pd { \la B + I^* \ra } _\La$, which is at most equal to $ \fpd \La ^{ \! \mathrm{op }} $ by definition, completing the proof of the last claim.
	\end{proof}

	\begin{rem}
			\label{rem:1}
		Note that $ \pd { \La  \! ^* } _\La $ is equal to the supremum of the injective dimensions of simple $ \La $-modules $ S $ with projective cover $ P $ whose Loewy length is non-maximal, according to the proof of \cref{thm:q.unif.Ll}{}. In \cref{con:q.unif.Ll.1}{}, we have labeled the vertices of $ Q $ so that these simple $ \La $-modules are exactly $ S_\La ( { i } ) $ for $ i = 1 ,2 \ldots , m $. Since there is a finite number of simple $ \La $-modules up to isomorphism, it is evident now that $ \i $ is of finite projective dimension (i.e.\ $ \pd { \La  \! ^* } _\La < \infty $) precisely when the injective dimension of every simple $ \La $-module $ S $ with $ \Ll( P ) $ non-maximal is finite. Equivalently, the same property holds precisely when $ \id_\La S $ being infinite implies maximality of $ \Ll ( P ) $ for every $ S $, i.e.\ when the Loewy length of $ \La $ is quasi-uniform.
	\end{rem}

	Before giving a concrete example of a family of algebras with quasi-uniform Loewy length, we consider the following conditions for a bound quiver algebra $\La = k Q / I $. If $ \La $ fails to satisfy any of these conditions, then the finiteness of its (big) finitistic dimension is either immediate or follows from the cited paper.
		\begin{enumerate}[\rm(i)]
			\item $ \La  $ has infinite global dimension.
			\item $\La $ has non-zero finitistic dimensions.
			\item $\La $ is not Iwanaga-Gorenstein (\!\!\cite{applications.of.contrvar.finite}{}).
			\item $\La$ is not monomial (\!\!\cite{GKK}{}). 
			\item The Loewy length of $\La$ is greater than three (\!\!\cite{GH}{}).
			\item $\La$ is triangular reduced (\!\!\cite{FGR}{}).
			\item The projective dimension of every simple $\La$-module is greater than one (\!\!\cite{FS}{}).
			\item Every arrow of $\La $ occurs in every generating set for $I$ (\!\!\cite{arrowrem1}{}).
		\end{enumerate}
	
	Recall that $ \La = k Q / I $ is \emph{triangular reduced} \cite{arrowrem1}{} if for every idempotent $e \neq 0 , 1$ both $e \La ( 1 - e ) $ and $(1 - e ) \La e $ are non-trivial subspaces of $\La$. Moreover, it holds that $\La $ is triangular reduced if and only if $Q$ is strongly connected, i.e.\ for every pair of vertices $v, v' \in Q_0 $ there is a path with source $v $ and target $v'$; see \cite[Lemma~3.12]{Giata2}{}.
	
	The next example contains an infinite family of bound quiver algebras with quasi-uniform Loewy length that satisfy the above conditions. In particular, the finiteness of the big finitistic dimension of the algebras follows from \cref{thm:q.unif.Ll}{}, and cannot be derived through the methods of the aforementioned papers.

	\begin{exam}
		\label{exam:q.unif.Ll}
		Fix two positive integers $n$ and $m$, where $n$ is a multiple of $5$ and at least equal to $10$. Let $ \La_{n, m} = k Q_{n, m } / I_{n , m}$ be the bound quiver algebra where $Q_{n,m}$ is the quiver of \hyperlink{fig'}{\cref{fig}{}}, the ideal $I_{n, m}$ is the ideal generated by relations $R_{n, m}$ in the table of the same figure, and $k$ is any field. For every vertex $i$ of $Q_{n, m}$, there is a loop $\lambda_i$ at $i$ exactly when $i = n+ m$ or $1 \leq i \leq n$ and $i $ is equivalent to $0$ or $4$ modulo $5$. Furthermore, if $1 \leq i < n $ then there is a unique arrow with source $i$ and target $i+1$, the arrow $\a_i $, except for $i = 5 $, and the relation $ \lambda_i \a_i^*$ is to be interpreted as both $\lambda_5 \a_5^1$ and $\lambda_5 \a_5^2$ if $i = 5$ or, else, as just the relation $\lambda_i \a_i $.

		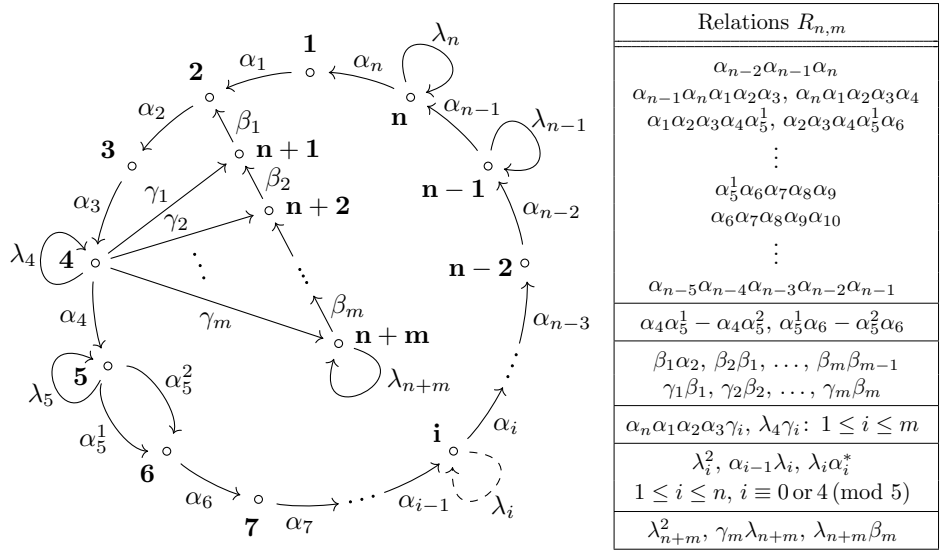
\begin{figure}[h!]		\phantomsection		\hypertarget{fig'}{}
			\vspace*{0.6cm}
			\centering
			\begin{minipage}{0.62\textwidth}
				\centering
				\begin{tikzpicture}[scale=0.65]


					\foreach \ang\lab in {90/1, 117.7/2, 145.4/3, 173.1/4, 200.8/5, 228.5/6 , 256.2/7 }{
						\draw  ($(0,0)+(\ang:4.4)$) circle (.08);
						\node at ($(0,0)+(\ang:4.99)$) {$\mathbf{\lab}$};
					}


					\draw ($(0,0)+(311.55:4.4)$) circle (.08);
					\node at ($(0,0)+(311.55:3.9)$) {$\mathbf{i}$};
					
					\draw ($(0,0)+(62.3:4.4)$) circle (.08);
					\node at ($(0,0)+(62.3:3.9)$) {$\mathbf{n}$};
					
					\draw ($(0,0)+(34.6:4.4)$) circle (.08);
					\node at ($(0,0)+(34.6:3.55)$) {$\mathbf{n-1}$};
					
					\draw ($(0,0)+(6.9:4.4)$) circle (.08);
					\node at ($(0,0)+(6.9:3.5)$) {$\mathbf{n-2}$};


					\draw ($(0,0)+(117.7:3.1)$) circle (.08);
					\node at ($(0,0)+(117.7:3.1)+(5:1)$) {$\mathbf{n+1}$};
					
					\draw ($(0,0)+(117.7:1.8)$) circle (.08);
					\node at ($(0,0)+(117.7:1.8)+(5:1)$) {$\mathbf{n+2}$};

					\draw[fill=black] ($(0,0)+(117.7:0.45)$) circle (.02);
					\draw[fill=black] ($(0,0)+(117.7:0.3)$) circle (.02);
					\draw[fill=black] ($(0,0)+(117.7:0.15)$) circle (.02);
					
					\draw ($(0,0)+(297.7:1.25)$) circle (.08);
					\node at ($(0,0)+(297.7:1.25)+(8:1.1)$) {$\mathbf{n+m}$};


					\foreach \ang in {280.75 , 283.85 , 286.95 , 336.15 , 339.25 , 342.35}{
						\draw[fill=black] ($(0,0)+(\ang:4.4)$) circle (.02);
					}


					\foreach \ang\lab in { 117.7/2 , 145.4/3 , 173.1/4 , 228.5/6 }{
						\draw[->,shorten <=7pt, shorten >=7pt] ($(0,0)+(\ang:4.4)$) arc (\ang:\ang+27.7:4.4);
						\node at ($(0,0)+(\ang+13.85:4.84)$) {$\alpha_{\lab}$};
					}
					
					\draw[->,shorten <=7pt, shorten >=7pt] ($(0,0)+(90:4.4)$) arc (90:117.7:4.4);
					\node at ($(0,0)+(104:4.75)$) {$\alpha_1$};
					\draw[->,shorten <=7pt, shorten >=7pt] ($(0,0)+(62.3:4.4)$) arc (62.3:90:4.4);
					\node at ($(0,0)+(75.4:4.75)$) {$\alpha_n$};
					\draw[->,shorten <=7pt, shorten >=7pt] ($(0,0)+(34.6:4.4)$) arc (34.6:62.3:4.4);
					\node at ($(0,0)+(48:4.95)$) {$\alpha_{n-1}$};
					\draw[->,shorten <=7pt, shorten >=7pt] ($(0,0)+(6.9:4.4)$) arc (6.9:34.6:4.4);
					\node at ($(0,0)+(17.7:5.16)$) {$\alpha_{n-2}$};


					\draw[->,shorten <=7pt] ($(0,0)+(256.2:4.4)$) arc (256.2:277.65:4.4);
					\draw[->,shorten >=7pt] ($(0,0)+(290.05:4.4)$) arc (290.05:311.55:4.4);
					\draw[->,shorten <=7pt] ($(0,0)+(311.55:4.4)$) arc (311.55:333.05:4.4);
					\draw[->,shorten >=7pt] ($(0,0)+(345.45:4.4)$) arc (345.45:366.9:4.4);
					\node at ($(0,0)+(267.5:4.8)$) {$\alpha_7$};
					\node at ($(0,0)+(298:4.96)$) {$\alpha_{i-1}$};
					\node at ($(0,0)+(325:4.87)$) {$\alpha_i$};
					\node at ($(0,0)+(352.5:5.17)$) {$\alpha_{n-3}$};


					\draw[->,shorten <=7pt, shorten >=7pt] ($(0,0)+(200.8:4.4)$) to [bend left=55]  ($(0,0)+(228.5:4.4)$) ;
					\draw[->,shorten <=7pt, shorten >=7pt] ($(0,0)+(200.8:4.4)$) to [bend right=55]  ($(0,0)+(228.5:4.4)$) ;
					
					\node at ($(0,0)+(214.65:3.23)$) {$\alpha_5^2$};
					\node at ($(0,0)+(214.65:5.35)$) {$\alpha_5^1$};


					\draw[->,shorten <=5pt, shorten >=5pt] ($(0,0)+(117.7:3.1)$) to   ($(0,0)+(117.7:4.4)$) ;
					\node at ($(0,0)+(117.7:3.75)+(7:0.5)$) {$\beta_1$};
					
					\draw[->,shorten <=5pt, shorten >=5pt] ($(0,0)+(117.7:1.8)$) to   ($(0,0)+(117.7:3.1)$) ;
					\node at ($(0,0)+(117.7:2.45)+(7:0.5)$) {$\beta_2$};
					
					\draw[->, shorten >=5pt] ($(0,0)+(117.7:0.75)$) to   ($(0,0)+(117.7:1.8)$) ;
					
					\draw[->,shorten <=5pt] ($(0,0)+(297.7:1.25)$) to   ($(0,0)+(297.7:0.15)$) ;
					\node at ($(0,0)+(297.7:0.5)+(10:0.55)$) {$\beta_m$};


					\draw[->,shorten <=6pt, shorten >=6pt] ($(0,0)+(173.1:4.4)$) to   ($(0,0)+(117.7:3.1)$) ;
					
					\draw[->,shorten <=6pt, shorten >=6pt] ($(0,0)+(173.1:4.4)$) to   ($(0,0)+(117.7:1.8)$) ;
					
					\draw[->,shorten <=6pt, shorten >=6pt] ($(0,0)+(173.1:4.4)$) to   ($(0,0)+(297.7:1.25)$) ;

					\path 
					($(0,0)+(173.1:4.4)$)	coordinate (v4)      
					($(0,0)+(117.7:3.1)$)	coordinate (vn+1)
					($(0,0)+(117.7:1.8)$) 	coordinate (vn+2)
					($(0,0)+(297.7:1.25)$)	coordinate (vn+m)
					;
					
					\coordinate (Mid v4 vn+1) at ($(v4)!0.48!(vn+1)$);
					\coordinate (Mid v4 vn+2) at ($(v4)!0.49!(vn+2)$);
					\coordinate (Mid v4 vn+m) at ($(v4)!0.51!(vn+m)$);

					\node at ($(Mid v4 vn+1)+(115:0.35)$) {$\gamma_1$};
					
					\node at ($(Mid v4 vn+2)+(105:0.3)$) {$\gamma_2$};
					
					\draw[fill=black]  ($(Mid v4 vn+2)+(310:0.45)$) circle (.02);
					\draw[fill=black]  ($(Mid v4 vn+2)+(305:0.7)$) circle (.02);
					\draw[fill=black]  ($(Mid v4 vn+2)+(300:0.95)$) circle (.02);
					
					\node at ($(Mid v4 vn+m)+(260:0.35)$) {$\gamma_m$};


					\draw[->,shorten <=6pt, shorten >=6pt] ($(0,0)+(173.1:4.4)$).. controls +(173.1+45:2) and +(173.1-45:2) .. ($(0,0)+(173.1:4.4)$);
					\node at ($(0,0)+(173.1:5.88)$) {$\lambda_4$};
					
					\draw[->,shorten <=6pt, shorten >=6pt] ($(0,0)+(200.8:4.4)$).. controls +(200.8+41:2) and +(200.8-45:2) .. ($(0,0)+(200.8:4.4)$);
					\node at ($(0,0)+(201.4:5.87)$) {$\lambda_5$};
					
					\draw[->,shorten <=6pt, shorten >=6pt, dashed] ($(0,0)+(311.55:4.4)$).. controls +(311.55+45:2) and +(311.55-45:2) .. ($(0,0)+(311.55:4.4)$);
					\node at ($(0,0)+(311.55:5.92)$) {$\lambda_i$};
					
					\draw[->,shorten <=6pt, shorten >=6pt] ($(0,0)+(62.3:4.4)$).. controls +(62.3+45:2) and +(62.3-45:2) .. ($(0,0)+(62.3:4.4)$);
					\node at ($(0,0)+(61.9:5.88)$) {$\lambda_n$};
					
					\draw[->,shorten <=6pt, shorten >=6pt] ($(0,0)+(34.6:4.4)$).. controls +(34.6+45:2) and +(34.6-45:2) .. ($(0,0)+(34.6:4.4)$);
					\node at ($(0,0)+(34:6.12)$) {$\lambda_{n-1}$};
					
					\draw[->,shorten <=6pt, shorten >=6pt] ($(0,0)+(297.7:1.25)$).. controls +(297.7+45:2) and +(297.7-45:2) .. ($(0,0)+(297.7:1.25)$);
					\node at ($(0,0)+(320:2.9)$) {$\lambda_{n+m}$};

				\end{tikzpicture}	
			\end{minipage}\hfill
			\begin{minipage}{0.36\textwidth}
				\centering
				\resizebox{4.5cm}{!}{
					\begin{tabular}{|c|}
						\hline
						\TTT	Relations $ R_{n, m} $	\BBB \\
						\hhline{|=|}
						
						\TTT	$\a_{n-2} \a_{n-1} \a_n$    \\
						$\a_{n-1} \a_n \a_1 \a_2 \a_3  , \, \a_n \a_1 \a_2 \a_3 \a_4$ \\
						$\a_1 \a_2 \a_3 \a_4 \a_5^1  , \, \a_2 \a_3 \a_4 \a_5^1 \a_6$ \\
						$\vdots$ \\
						$ \a_5^1 \a_6 \a_7 \a_8 \a_9 $ \\
						$\a_6 \a_7 \a_8 \a_9 \a_{10}$  \\
						$\vdots$  \\
						$ \a_{n-5} \a_{n-4} \a_{n-3} \a_{n-2} \a_{n-1}$	 \BBB  \\
						\hline	
						\TTT	$\a_4 \a_5^1 - \a_4 \a_5^2  , \, \a_5^1 \a_6 - \a_5^2 \a_6$	\BBB	\\
						\hline
						\TTT	$\b_1 \a_2  , \, \b_2 \b_1  , \, \ldots  , \, \b_m \b_{m-1}$ \\
						$ \g_1 \b_1  , \, \g_2 \b_2  , \, \ldots  , \, \g_m \b_m $	\BBB \\
						\hline
						\TTT	$ \a_n \a_1 \a_2 \a_3 \g_i , \, \lambda_4 \g_i	
						\colon \, 1 \leq i \leq m $
						\BBB \\ 
						\hline
						\TTT
						$\lambda_i^2   , \,  \a_{i-1} \lambda_i   , \,  \lambda_i \a_i^*  $	\BBB	\\
						$ 1 \leq i \leq n  , \, i \equiv 0 \, \textrm{or} \, 4 \, ( \!\!\!\! \mod 5 ) $
						\BBB	\\
						\hline
						\TTT
						$ \lambda_{n+m}^2 , \,  \gamma_m \lambda_{n+m} , \,  \lambda_{n+m} \beta_m  $
						\BBB \\
						\hline
					\end{tabular}
				}
			\end{minipage}
			\caption{Quiver and relations for $ \La_{n, m} $}
			\label{fig}
			\vspace*{0.2cm}
		\end{figure}

		Our aim is to establish that $ \Fpd \La < \infty $ for $\La = \La_{n,m} $, by showing that the Loewy length of $ \La $ is quasi-uniform. Note that $\La $ is a bound quiver algebra with $\Ll( \La ) = 5 $ and the only indecomposable direct summands of ${}_\La \La $ with non-maximal Loewy length are the ones corresponding to vertices $1$ and $2$. Indeed, we have $\Ll ( \La e_1  ) = 3$ and $\Ll ( \La e_2 ) = 4 $ as the paths $ \a_{n-1} \a_n $ and $ \a_{n-1} \a_n \a_1 $ are non-zero, while the paths $\a_{n-2} \a_{n-1} \a_n $, $ \lambda_n \a_n $, $ \lambda_{n-1} \a_{n-1} $, $ \g_1 \b_1 $ and $ \b_2 \b_1 $ are zero. Therefore,
		\[
		\Fpd \La \leq \max \{ \idim S_\La ( {1} ) , \idim S_\La ( {2} ) \}
		\]
		according to \cref{thm:q.unif.Ll}{}, as $ f_{ \mathsf{n.m}} = e_1 + e_2 $.

		The algebra $\La $ satisfies conditions (i) to (viii) preceding the example. Conditions (i) and (vi) follow immediately from the shape of $Q_{n ,m }$ (see \cite{ILP}{} for (i), and \cite[Lemma~3.12]{Giata2}{} for (vi)) and $\La $ is not monomial as the paths $\a_4 \a_5^1 $ and $\a_4 \a_5^2 $ are equal and non-zero in $\La $. Furthermore, it holds that the finitistic dimensions of $\La $ are non-zero as the regular module $\La_\La $ does not possess a submodule that is isomorphic to $S_{\La ^{ \! \mathrm{op}}} ( {1} ) $, see \cite[Lemma~6.2]{Bass}{}. Thirdly, it is quite straightforward to verify that every arrow of $Q_{n,m}$ occurs in every generating set for $I_{n,m}$, and the algebra $\La $ is not Iwanaga-Gorenstein as \cite[Corollary~4.14]{Giata2}{} applies for the loop $\lambda_{n+m}$ and the arrow $\gamma_m$ (or the arrow $\beta_m$).
		
		To see that condition (vii) is also satisfied, note first that $S_\La ( {i} ) $ has infinite projective dimension for every vertex $i$ such that \mbox{$1 \leq i \leq n $} and $i \equiv 0 \, (\!\!\!\! \mod 5 ) $ or $i \equiv 4 \, (\!\!\!\! \mod 5 ) $, due to the loop $\lambda_i$. If we let vertex $n$ be represented also by $ 0 $, then for every vertex $i$ such that $1 \leq i \leq n $ and $i \equiv 1 \, (\!\!\!\! \mod 5 ) $, it holds that the module $S_\La ( {i-1} )$ is a direct summand of $\Omega^2 ( S_\La ( {i} ) )$. Similarly, the projective dimension of $  S_\La ( {n + m} ) $ is infinite due to the loop $\lambda_{n+m}$ and, if $i$ is such that $ 1 \leq i \leq m-1 $, then the module $S_\La( {n+ i + 1 } )$ is a direct summand of $\Omega^1 ( S_\La ( {n+ i } ) )$. Finally, the module $S_\La ( { n+1 } )$ is a direct summand of $ \Omega^1 ( S_\La ( { 2 } ) ) $ and $ \Omega^2 ( S_\La ( { 3 } ) ) $ and, 
		for every vertex $i$ such that $7 \leq i \leq n$ and $i \equiv 2 \, (\!\!\!\! \mod 5 ) $ or $i \equiv 3 $ $ (\!\!\!\! \mod 5 ) $, the module $S_\La ( {i - 5} )$ is a direct summand of $\Omega^2 ( S_\La ( { i } ) ) $. We conclude that all simple $\La$-modules have infinite projective dimension.

		To compute the upper bound for $ \Fpd \La $ provided by \cref{thm:q.unif.Ll}{}, we check that there are minimal projective resolutions of the form
		\[
		0 \to e_{n-3} \La \to e_{n-4} \La  \to \cdots \to e_7 \La \to  e_6 \La \to e_2 \La \to e_1 \La \to S_{\La^{ \! \mathrm{op}}} ( 1 ) \to 0
		\]
		and
		\[
		0 \to e_{n-2} \La \to e_{n-3} \La  \to \cdots \to e_8 \La \to  e_7 \La \to e_3 \La \to e_2 \La \to S_{\La^{ \! \mathrm{op}}} ( 2 ) \to 0
		\]
		implying that $ \idim S_\La ( {1} ) = \idim S_\La ( {2} ) =  2 \nu - 1$, where $\nu = \frac{n}{5} $, due to the standard duality between finitely generated left and right $\La $-modules. All in all, we deduce that $ 1 \leq \Fpd \La \leq 2 \nu - 1 $.
		
		The software \cite{QPA}{} was used in order to verify that a preliminary version of the algebra $\La_{10,0}$ had some of the properties required for the purpose of this example out of many other candidate algebras.
	\end{exam}

	\begin{rem}
		One can show that $ \fpd \La^{ \! \mathrm{op}} $ is also finite for $ \La = \La_{n , m} $ of \cref{exam:q.unif.Ll}{} by a successive application of \cite[Proposition~2.1]{FS}{} (see also \cite[Theorem~5.5]{arrowrem1}{}). In other words, we create a sequence of algebras $ \La = \La_1 , \La_2 , \ldots , \La_\nu $ (for $ \nu = \frac{n}{5} $) by setting $ \La_{ i+1 } = (1 - f_i) \La_{ i } (1 - f_i) $ for every $ i = 1, 2, \ldots , \nu-1 $, where $ f_i $ is the sum of trivial paths in $ \La_{ i } $ such that the corresponding simple right $ \La_i $-modules have projective dimension at most one. Then \cite[Proposition~2.1]{FS}{} ensures that $ \fpd \La_{ i+1 }^{\! \mathrm{op} } < \infty $ implies $ \fpd \La_{ i }^{ \! \mathrm{op} } < \infty $ for every $ i $. Specifically, one can show that $ f_i = e_{n - 5i + 1 } + e_{n - 5i + 2 } $ for every $ i $ through \cite{ILP}{} and \cite[Lemma~4.11]{Giata2}{}, whence the quiver $ Q_{ i+1 } $ of each algebra $ \La_{ i+1 } $ results from $ Q_i $ by removing vertices $ { n - 5i + 1 } $ and $ { n - 5i + 2 } $, and all their adjacent arrows. However, we also have to add a new `connecting' arrow $ \d_i \colon { n - 5i } \to { n - 5i + 3 } $, which preserves the triangular reduced nature of the produced algebras, and adjust the relations determining $ \La_{ i+1 } $ accordingly from the relations determining $ \La_i $. In particular, algebra $ \La_\nu $ is monomial as the arrows $ \a_5^1 $ and $ \a_5^2 $ (along with the arrows $ \a_6 $ and $ \a_7 $) are replaced by the connecting arrow $ \d_{ \nu-1 } \colon { 5 } \to { 8 } $, which gets rid of the non-monomial relations $ \a_4 \a_5^1 - \a_4 \a_5^2 $ and $ \a_5^1 \a_6 - \a_5^2 \a_6 $ present in all the previous algebras of the sequence. We conclude that $ \fpd \La^{ \! \mathrm{op}} < \infty $, since the finitistic dimensions of monomial algebras are finite by \cite{GKK}{}.
	\end{rem}
	
	If $\La = \La_{n , m}$ is the algebra of the above example, then there is an exact sequence of the form
	\[
	0 \to S_{\La^{ \! \mathrm{op}}} ( 2 ) \to e_{n+1} \La \to e_{n+2} \La \to \cdots \to e_{n+m -1} \La \to S_{\La^{ \! \mathrm{op}}} ( {n+m-1} ) \to 0
	\]
	which is a truncation of the minimal projective resolution of $S_{\La^{ \! \mathrm{op}}} ( {n + m - 1} )$. In particular, it holds that $\idim S_\La ( { n + m - 1 } ) = 2 \nu + m - 2 $. We close this section with two observations based on this fact.
	
	Firstly, there are bound quiver algebras $\La $ with quasi-uniform Loewy length such that the supremum of all finite injective dimensions of simple modules is arbitrarily bigger than the upper bound for $\Fpd \La $ provided by \cref{thm:q.unif.Ll}{}. In particular, the upper bound for $\fpd \La$ provided by the vertex removal operation \cite{arrowrem1}{} can be arbitrarily bigger than the upper bound provided by \cref{thm:q.unif.Ll}{}.

	Secondly, the difference $\fpd \La ^{ \! \mathrm{op}} - \Fpd \La $ can be arbitrarily big for non-monomial bound quiver algebras $\La $ with non-zero finitistic dimensions.
	
	\begin{cor}
			\label{cor:q.unif.Ll.1}
		It holds that $ \fpd  \La_{n , m} ^{ \! \mathrm{op}} - \Fpd \La_{n , m} \geq m - 1 $.
	\end{cor}
	
	\begin{proof}
		The corollary follows from \cref{exam:q.unif.Ll}{} and the above discussion as
			\[
				\Fpd \La_{n , m} \leq 2 \nu - 1  \leq 2 \nu + m - 2   \leq	\fpd \La_{n , m} ^{ \! \mathrm{op} }
			\]
		where the last inequality follows from $ 2 \nu + m - 2 = \pd  S_{\La_{n, m}^{ \! \mathrm{op} }}({n+m-1}) $.
	\end{proof}

\end{document}